\documentclass[12pt,english]{article}
\usepackage[latin9]{inputenc}
\usepackage{geometry}
\usepackage{booktabs}
\usepackage{amsthm}
\usepackage{amsmath}
\usepackage{amssymb}
\usepackage{graphicx}

\makeatletter

\providecommand{\tabularnewline}{\\}

\numberwithin{equation}{section}
\numberwithin{figure}{section}
\theoremstyle{plain}
\newtheorem{thm}{\protect\theoremname}[section]
  \theoremstyle{plain}
  \newtheorem{cor}[thm]{\protect\corollaryname}
 \theoremstyle{definition}
 \newtheorem*{defn*}{\protect\definitionname}
  \theoremstyle{remark}
  \newtheorem{rem}[thm]{\protect\remarkname}
  \theoremstyle{plain}
  \newtheorem{lem}[thm]{\protect\lemmaname}
  \theoremstyle{remark}
  \newtheorem*{rem*}{\protect\remarkname}
  \theoremstyle{remark}
  \newtheorem*{notation*}{\protect\notationname}
  \theoremstyle{plain}
  \newtheorem{prop}[thm]{\protect\propositionname}
  \theoremstyle{definition}
  \newtheorem*{example*}{\protect\examplename}

\usepackage{amsfonts}
\usepackage{mathdots}
\usepackage{asymptote}
\usepackage{tikz}
\usepackage{afterpage}
\usepackage{booktabs}
\usepackage{stmaryrd}

\newcommand{\CollatzGraph}{\mathcal{G}}
\newcommand{\overmod}[1]{\underset{#1}{\equiv}}
\DeclareMathOperator{\lcm}{lcm}
\DeclareMathOperator{\Greedy}{Greedy}
\DeclareMathOperator{\AGL}{AGL}
\newcommand{\rat}{\rightarrowtriangle}

\title{Strongly sufficient sets and the distribution of \\ arithmetic sequences in the $3x+1$ graph }
\author{
Keenan Monks 
\thanks{email: \texttt{monks@college.harvard.edu}} \\
Harvard University \\ Cambridge, MA 02138
\and
Kenneth G. Monks \thanks{email: \texttt{monks@scranton.edu} ; Corresponding author} \\
University of Scranton \\ Scranton, PA 18510
\and
Kenneth M. Monks \thanks{email: \texttt{monks@math.colostate.edu}} \\
Colorado State University \\ Fort Collins, CO 80523
\and
Maria Monks \thanks{email: \texttt{monks@math.berkeley.edu}} \\
University of California \\
Berkeley, CA 94720
}

\date{\today}

\global\long\def\xcong#1{\underset{#1}{\equiv}}

\global\long\def\bbN{\mathbb{N}}
\global\long\def\bbZ{\mathbb{Z}}

\makeatother

\usepackage{babel}
  \providecommand{\corollaryname}{Corollary}
  \providecommand{\definitionname}{Definition}
  \providecommand{\examplename}{Example}
  \providecommand{\lemmaname}{Lemma}
  \providecommand{\notationname}{Notation}
  \providecommand{\propositionname}{Proposition}
  \providecommand{\remarkname}{Remark}
\providecommand{\theoremname}{Theorem}

\begin{document}
\renewcommand{\thefootnote}{\fnsymbol{footnote}}
\renewcommand{\thefootnote}{\arabic{footnote}}  \maketitle
\begin{abstract}
The $3x+1$ Conjecture asserts that the $T$-orbit of every positive
integer contains $1$, where $T$ maps $x\mapsto x/2$ for $x$ even
and $x\mapsto(3x+1)/2$ for $x$ odd. A set $S$ of positive integers
is \textit{sufficient} if the orbit of each positive integer intersects
the orbit of some member of $S$. In \cite{Kenny} it was shown that
every arithmetic sequence is sufficient. 

In this paper we further investigate the concept of sufficiency. We
construct sufficient sets of arbitrarily low asymptotic density in
the natural numbers. We determine the structure of the groups generated
by the maps $x\mapsto x/2$ and $x\mapsto(3x+1)/2$ modulo $b$ for
$b$ relatively prime to $6$, and study the action of these groups
on the directed graph associated to the $3x+1$ dynamical system.
From this we obtain information about the distribution of arithmetic
sequences and obtain surprising new results about certain arithmetic
sequences. For example, we show that the forward $T$-orbit of every
positive integer contains an element congruent to $2\bmod9$, and
every non-trivial cycle and divergent orbit contains an element congruent
to $20\bmod27$. We generalize these results to find many other sets
that are strongly sufficient in this way. 

Finally, we show that the $3x+1$ digraph exhibits a surprising and
beautiful self-duality modulo $2^{n}$ for any $n$, and prove that
it does not have this property for any other modulus. We then use
deeper previous results to construct additional families of nontrivial
strongly sufficient sets by showing that for any $k<n$, one can ``fold''
the digraph modulo $2^{n}$ onto the digraph modulo $2^{k}$ in a
natural way. 
\end{abstract}

\section{Introduction}

The $3x+1$ Conjecture, also known as the Collatz Conjecture, is a
famous open problem in discrete dynamics. Attributed to L. Collatz
in the 1930's, the conjecture states that if $T:\mathbb{Z}\to\mathbb{Z}$
is defined to be 
\[
T(x)=\begin{cases}
\frac{x}{2} & x\text{ is even}\\
(3x+1)/2 & x\text{ is odd}
\end{cases},
\]
then for any positive integer $x$, there is a nonnegative integer
$k$ for which $T^{k}(x)=1$. In other words, the \textit{$T$}\textit{\emph{-orbit}}
of $x$ (i.e. the sequence $x,T(x),T(T(x)),\ldots$) contains the
number $1$ among its elements.

Historically, the Collatz problem has been broken down into two smaller
conjectures: 
\begin{quote}
\textbf{Nontrivial Cycles Conjecture:} There are no $T$-cycles of
positive integers other than the cycle containing $1$. 
\end{quote}

\begin{quote}
\textbf{Divergent Orbits Conjecture:} There are no divergent $T$-orbits
of positive integers. 
\end{quote}
Together, these two statements suffice to prove the $3x+1$ conjecture.
Both remain unresolved.

There has recently been progress towards reducing the $3x+1$ problem
to a seemingly simpler problem. We say that positive integers $x$
and $y$ \textit{merge} if there exist nonnegative integers $k$ and
$j$ for which $T^{k}(x)=T^{j}(y)$. A set of positive integers $S$
is said to be \textit{sufficient} if for every positive integer $x$,
there is an element $y\in S$ that merges with $x$. Notice that to
prove the $3x+1$ conjecture, it suffices to show that the $T$-orbit
of every element of some sufficient set $S$ contains $1$ since every
integer that merges with an element of $S$ must also enter the cycle
$\overline{1,2}$ as well. In \cite{Kenny}, the third author shows
that every arithmetic sequence is sufficient.

We can visualize these notions by drawing a directed graph associated
to the dynamical system $T:\mathbb{Z}\to\mathbb{Z}$. Let $T_{0}(x)=x/2$
and $T_{1}(x)=(3x+1)/2$, and define the $3x+1$\textit{ graph} $\mathcal{G}$
to be the two-colored directed graph on the positive integers having
a black edge from $x$ to $z$ if $T_{0}(x)=z$, and a red edge (which
are also dashed in the images in this paper for the benefit of those
reading a black and white printout) from $x$ to $z$ if $T_{1}(x)=z$.
(See Figure \ref{Fig:G}.) Then two integers merge if and only if
they are in the same connected component of $\mathcal{G}$, and a
sufficient set is a set of nodes which intersects every connected
component of $\mathcal{G}$. The $3x+1$ conjecture is true if and
only if $\mathcal{G}$ is connected.

\begin{figure}[h]
\begin{centering}
\includegraphics{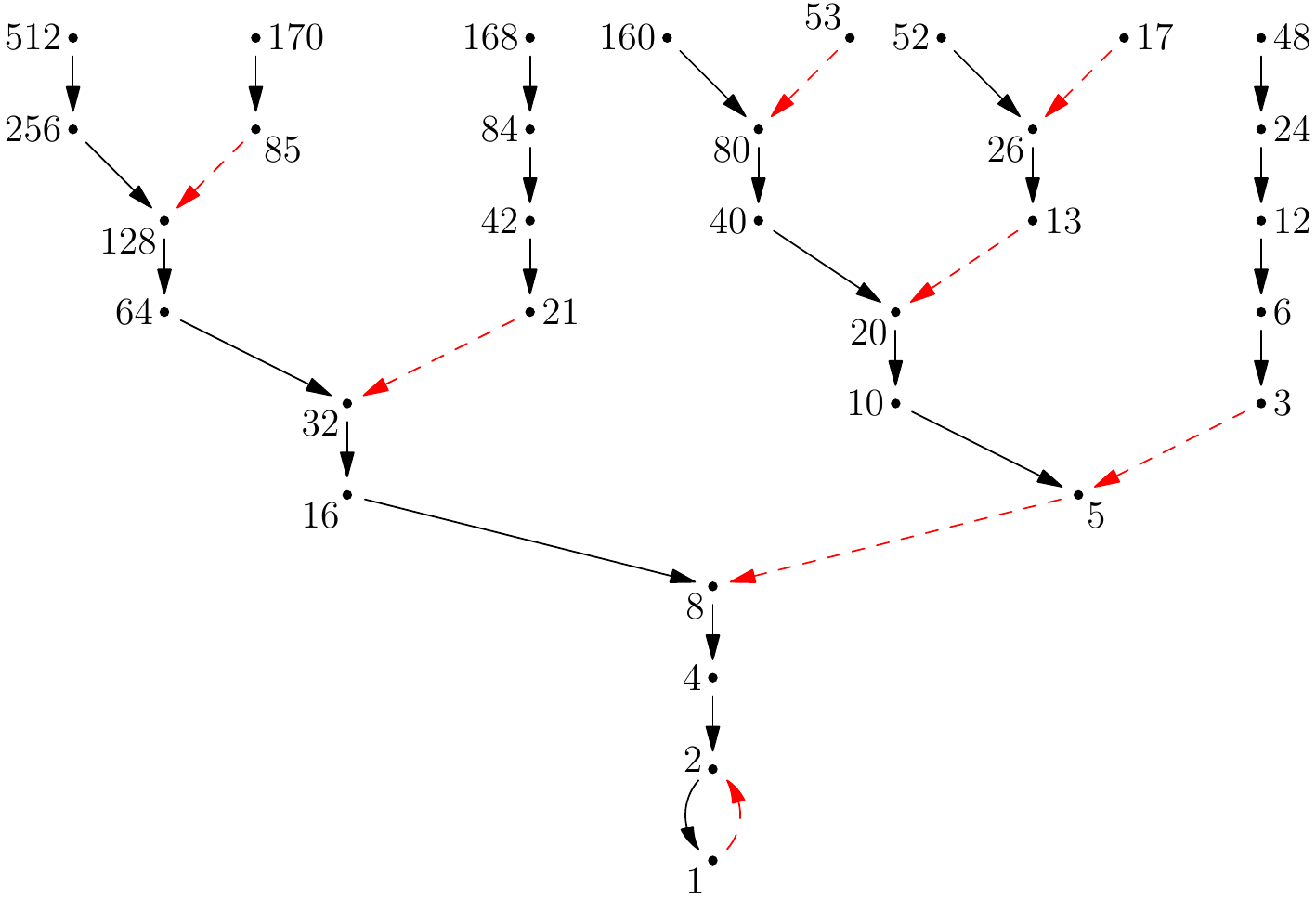} 
\par\end{centering}

\caption{\label{Fig:G} A portion of the $3x+1$ graph $\mathcal{G}$ near
$1$.}
\end{figure}

In this paper we undertake a deeper investigation into the distribution
of arithmetic sequences in the $3x+1$ graph and properties of sufficient
sets in general. Define a \textit{forward tracing path} to be a path
in $\mathcal{G}$ along the directed arrows (which is simply an initial
segment of a $T$-orbit), and a \textit{back tracing path} to be a
path in $\mathcal{G}$ against the direction of these arrows.

For brevity, we write ``$a\overmod{b}c$'' in place of ``$a\equiv c\pmod{b}$''
throughout. Using this notation, we recall the well-known fact that
if $x\overmod{3}0$, then the only way to form a back tracing path
starting from $x$ is by applying $T_{0}^{-1}$ repeatedly, forming
a single reverse chain of black arrows in $\mathcal{G}$. Moreover,
every element of this chain is also divisible by $3$. Conversely,
every positive integer $x$ forward traces to a number that is not
divisible by $3$, at which point all future points in its orbit are
not divisible by $3$. Thus, the back tracing paths from multiples
of $3$ are well-understood, and for convenience we also define the
\textit{pruned }$3x+1$\textit{ Graph}, denoted $\widetilde{\mathcal{G}}$,
to be the subgraph of $\mathcal{G}$ consisting of the positive integers
relatively prime to $3$. A portion of the $ $$\widetilde{\mathcal{G}}$
is illustrated in Figure \ref{Fig:GPruned}.

In \cite{Kenny}, in order to prove that every arithmetic sequence
is sufficient, the third author shows that for any positive integers
$a$ and $d$ and for every node $x$ in the pruned $3x+1$ graph
$\widetilde{\mathcal{G}}$, there is a back tracing path from $x$
to some $y$ in $\mathcal{G}$ with $y\in\{a+dn\mid n\in\mathbb{Z}\}$.
(Since any $x$ in the general $3x+1$ graph has a forward tracing
path to some $x'$ in the pruned graph, this shows that every $x$
in $\mathcal{G}$ merges with some $y$ in the given arithmetic sequence.)
In section \ref{efficient}, we strengthen these results by finding
bounds on the minimum number of red arrows needed to back trace from
any positive integer $x$ to an integer $y$ in a desired arithmetic
sequence. These results rely on an understanding of the $3x+1$ groups
$G_{b}$ generated by $T_{0}$ and $T_{1}$ modulo an integer $b$
relatively prime to $6$. We completely classify these groups in section
\ref{groups}.

For specific moduli $d$, we can obtain even stronger and more surprising
results. In section \ref{inverselimits}, we describe infinite back
tracing paths as elements of an inverse limit and show that every
such back tracing path must contain an integer congruent to $2$ mod
$9$. We use similar methods to show that in fact every $T$-orbit
of positive integers must also contain an element congruent to $2$
mod $9$. For this reason, we say that the set of integers congruent
to $2$ mod $9$ is \textit{strongly sufficient}.

In Section \ref{forward}, we define $\Gamma_{d}$ to be the finite
directed graph obtained by taking the $3x+1$ graph $\mathcal{G}$
mod $d$. We use these graphs to determine graph-theoretic criteria
for a set to be strongly sufficient, and we provide several families
of such sets $S$ that must intersect every nontrivial $T$-orbit
or infinite back tracing path. We also demonstrate that finding strongly
sufficient sets is a plausible way to approach the Nontrivial Cycles
Conjecture.

Finally, in Section \ref{DualityFolding}, we show that the graphs
$\Gamma_{2^{n}}$ exhibit a surprising and beautiful self-duality
(given by a map first defined in \cite{Dad}) and that these are the
only $\Gamma_{d}$ having this property. We also use results from
\cite{Maria} and \cite{Keenan} to show that for any $k<n$, one
can fold $\Gamma_{2^{n}}$ onto $\Gamma_{2^{k}}$ in a natural way.
We combine these deeper tools with our results on strong sufficiency
to obtain an infinite family of strongly sufficient sets consisting
of unions of residue classes modulo a power of $2$.

\section{Sparse sufficient sets}

Since every arithmetic sequence is sufficient, the arithmetic sequences
form a family of sufficient sets with members having arbitrarily small
positive density in the integers. It is natural to ask if there is
a sufficient set of density zero in the positive integers. We answer
this question in the affirmative as follows.
\begin{thm}
For any function $f:\mathbb{N}\to\mathbb{N}$ and any positive integers
$a$ and $d$, the set of integers 
\[
\{2^{f(n)}(a+dn)\mid n\in\mathbb{N}\}
\]
 is a sufficient set. \end{thm}
\begin{proof}
We know that the set $\{a+dn\mid n\in\mathbb{N}\}$ is sufficient.
So, given a positive integer $x$, there is a number of the form $a+dN$
that merges with $x$. Thus, the positive integer $2^{f(N)}(a+dN)$,
which maps to $a+dN$ after $f(N)$ iterations of $T$, also merges
with $x$. This completes the proof.
\end{proof}
By taking $f(n)$ to be sufficiently large relative to $n$, we can
use this to produce infinitely many sufficient sets of density zero
in the positive integers. We state one family of these below.
\begin{cor}
For any fixed $a$ and $d$, the sequence $(a+dn)\cdot2^{n}$ is a
sufficient set with asymptotic density zero in the positive integers.
\end{cor}
Thus, to prove the $3x+1$ conjecture, it suffices to show that the
$T$-orbit of every number in the density-zero set $\{(a+dn)\cdot2^{n}\mid n>0\}$
contains $1$. This method can also be used to find arbitrarily sparse
sufficient sets containing only odd numbers (for example, the set
$\left\{ \left(2^{2f(n)+1}(a+3dn)-1\right)/3\mid n\in\mathbb{N}\right\} $
for $a\xcong32$ and $f:\bbN\to\bbZ_{+}$). Notice, however, that
the elements of any such sufficient sets eventually map to $a+dn$,
so one still effectively needs to show that the elements of the arithmetic
sequence $\{a+dn\}$ map to $1$ under $T$. Thus we turn our attention
to this problem next by investigating the actual distribution of such
sequences in the $3x+1$ graph.

\section{Classification of the groups $G_{b}$}

\label{groups}

For any positive integer $b$ relatively prime to $6$, the functions
$T_{0}$ and $T_{1}$ act as permutations on $\mathbb{Z}/b\mathbb{Z}$.
We begin by completely classifying the permutation groups $G_{b}$
generated by these two permutations.

Let $C_{r}$ denote the cyclic group of order $r$. Also let 
\[
\AGL(1,b)=\{x\mapsto cx+d\mid d\in\mathbb{Z}/b\mathbb{Z},c\in\left(\mathbb{Z}/b\mathbb{Z}\right)^{*}\}
\]
be the group of affine maps mod $b$ under composition, and note that
$G_{b}$ can be viewed as the subgroup of $\AGL(1,b)$ generated by
$T_{0}(x)=x/2$ and $T_{1}(x)=\left(3x+1\right)/2$. Moreover, the
subgroup $\{x\mapsto x+a\mid a\in\mathbb{Z}/b\mathbb{Z}\}$ is isomorphic
to the cyclic group $C_{b}$, and is a normal subgroup of $\AGL(1,b)$.
It will also be useful to consider the element $P\left(x\right)=x+1$
of $G_{b}$.
\begin{thm}
\label{structure} Let $b$ be a positive integer relatively prime
to 2 and 3. Write the prime factorization $b=p_{1}^{e_{1}}p_{2}^{e_{2}}\cdots p_{n}^{e_{n}}$.
For all $i\in\{1,\ldots,n\}$, let $s_{i}$ be the multiplicative
order of $2\bmod p_{i}^{e_{i}}$, let $t_{i}$ be the multiplicative
order of $3\bmod p_{i}^{e_{i}}$, and let $a_{i}=\lcm(s_{i},t_{i})$.
Then 
\[
G_{b}\cong C_{b}\rtimes M
\]
 where 
\[
M=C_{a_{1}}\times C_{a_{2}}\times\cdots\times C_{a_{n}}.
\]
 
\end{thm}
In the statement above, the action on $C_{b}$ defining the semidirect
product is the action of $\AGL(1,b)$ by conjugation on the subgroup
\[
\{x\mapsto x+a\mid a\in\mathbb{Z}/b\mathbb{Z}\}\cong C_{b}.
\]

\begin{proof}
In \cite{Kenny}, it was shown that the function $P\left(x\right)=x+1$
is an element of the group generated by $T_{0}$ and $T_{1}$. This
function clearly has order $b$ and generates the cyclic subgroup
$C_{b}=\{x\mapsto x+d:d\in\mathbb{Z}/b\mathbb{Z}\}$. This is a normal
subgroup, since it is a fixed set under conjugation.

Since $x\mapsto x+1$ is in $G_{b}$, and $T_{0}$ and $T_{1}$ can
be expressed in terms of the maps $x\mapsto x+1$, $x\mapsto2x$,
and $x\mapsto3x$, we have that $G_{b}\subseteq\left<x\mapsto x+1,x\mapsto2x,x\mapsto3x\right>$
as a subgroup of $\AGL(1,b)$. Moreover, $2x$ and $3x$ can be generated
using $T_{0}$, $T_{1}$, and $x+1$, so in fact 
\[
G_{b}=\left<x\mapsto x+1,x\mapsto2x,x\mapsto3x\right>.
\]
The first generator corresponds to the cyclic subgroup $C_{b}$. We
now only need to see what we obtain from multiplication by $2$ and
$3$.

By the Chinese Remainder Theorem, we have $\mathbb{Z}/b\mathbb{Z}\cong C_{{p_{1}}^{e_{1}}}\times C_{{p_{2}}^{e_{2}}}\times\cdots\times C_{{p_{n}}^{e_{n}}}$.
Thus we can look at the action of multiplication by $2$ and $3$
on each component, and the action on the whole group $G_{b}$ will
be the direct product of each of these.

Let $p\in\{p_{1},p_{2},\ldots,p_{n}\}$ and $e$ be the corresponding
exponent. Since $b$ is relatively prime to $2$, we know that $p$
is an odd prime. Thus $\left(Z/p^{e}Z\right)^{\ast}$, the group of
units of $Z/p^{e}Z$, is cyclic. Let $s$ be the order of $2$ and
$t$ the order of $3$ in $\left(Z/p^{e}Z\right)^{\ast}$. Since the
subgroup lattice of $\left(Z/p^{e}Z\right)^{\ast}$ is isomorphic
to the divisor lattice of $\phi(p^{e})$, we have that $\|\left<2,3\right>\|=\lcm(s,t)$,
which concludes the proof.
\end{proof}
While Theorem \ref{structure} describes the overall structure of
the groups, it would be useful to understand it as a finitely generated
group in terms of the generators $T_{0}$ and $T_{1}$ (mod $b$).
We begin by calculating the orders of $T_{0}$ and $T_{1}$ in $G_{b}$.
To do so, we introduce the auxiliary function $E$.
\begin{defn*}
Let $E_{0}(x)=3x/2$ and $E_{1}(x)=(x+1)/2$, and define $E:\mathbb{Z}_{+}\to\mathbb{Z}_{+}$
by 
\[
E(x)=\begin{cases}
E_{0}(x) & x\text{ is even}\\
E_{1}(x) & x\text{ is odd}
\end{cases}.
\]
 
\end{defn*}
Let $P(x)=x+1$. Then straightforward calculation shows that $E=PTP^{-1}$,
and in particular that $E_{0}=PT_{1}P^{-1}$ and $E_{1}=PT_{0}P^{-1}$.
Thus, the $E$-orbit of a positive integer $x>1$ can be obtained
by taking the $T$-orbit of $x-1$ and adding $1$ to each element
of the orbit. Therefore the $3x+1$ conjecture is equivalent to showing
that the $E$-orbit of any positive integer $x>1$ contains the integer
$2$.
\begin{rem}
The conjugacy between $T$ and $E$ via $P$ makes computation of
orbits somewhat easier: to compute the $E$-orbit of a positive integer
$x>1$, one first replaces any $2$'s in the prime factorization of
$x$ with $3$'s, one at a time, until the result is odd. At that
point, one divides by $2$ and rounds up to the nearest integer, and
repeats the process. For instance, the $E$-orbit of $8$ is 
\[
8,12,18,27,14,21,11,6,9,5,3,2,\ldots,
\]
 which corresponds to the $T$-orbit of $7$: 
\[
7,11,17,26,13,20,10,5,8,4,2,1,\ldots.
\]
\end{rem}
\begin{lem}
\label{orders} Let $b$ be a positive integer relatively prime to
$2$ and $3$. The order of $T_{0}$ in $G_{b}$ is equal to the order
of $2$ modulo $b$, and the order of $T_{1}$ in $G_{b}$ is equal
to the order of $\frac{3}{2}$ modulo $b$.\end{lem}
\begin{proof}
The order of $T_{0}$ in $G_{b}$ is equal to the order of $1/2$
modulo $b$, which is equal to the order of $2$ modulo $b$.

For $T_{1}$, we have that $T$ is conjugate to $E$ on the positive
integers via the map $x\mapsto x+1$. Therefore, $T$ and $E$ are
also conjugate when considered as maps on $\mathbb{Z}/b\mathbb{Z}$.
In particular, the conjugacy sends $T_{0}$ to $E_{1}$ and $T_{1}$
to $E_{0}$.

Now, the order of $T_{1}$ in $G_{b}$ is equal to the order of $E_{0}$
in $G_{b}$ by the conjugacy, and the order of $E_{0}$ is equal to
the order of $\frac{3}{2}$ modulo $b$ (since $E_{0}(x)=\frac{3}{2}x$).
This completes the proof. 
\end{proof}
In \cite{Kenny}, it was shown that $G_{b}$ acts transitively on
$\mathbb{Z}/b\mathbb{Z}$, by showing that the map $P\left(x\right)=x+1$
is in $G_{b}$. It is easy to check that the shortest representation
for $P$ map in terms of the generators $T_{0}$ and $T_{1}$ is 
\begin{equation}
T_{0}^{-2}T_{1}T_{0}T_{1}^{-1}T_{0}=P\label{eq:x+1}
\end{equation}
and the corresponding result for the map $E$ is

\[
E_{1}^{-2}E_{0}E_{1}E_{0}^{-1}E_{1}=P.
\]

\section{Uniformity in the distribution of arithmetic sequences}

\label{efficient}

For this section and the next, we require some notation and basic
results involving back tracing. We generally follow the notation used
in Wirsching's book \cite{Wirsching}.

Define the set of \textit{feasible vectors} to be 
\[
\mathcal{F}=\underset{k=0}{\overset{\infty}{\bigcup}}\mathbb{N}^{k+1}.
\]
 Let $s\in\mathcal{F}$. Then $s=\left(s_{0},s_{1},\ldots,s_{k}\right)$
for some nonnegative integers $k$ and $s_{0},s_{1},\ldots,s_{k}$.
The \textit{length} of $s$, written $l\left(s\right)$, is $k$.
The norm of $s$, written $\left\vert \left\vert s\right\vert \right\vert $,
is $l\left(s\right)+\underset{i=0}{\overset{l\left(s\right)}{\sum}}s_{i}$.

For $s\in\mathcal{F\ }$ with $s=\left(s_{0},s_{1},\ldots,s_{k}\right)$,
Wirsching calls the function $v_{s}:\mathbb{Z}^{+}\rightarrow\mathbb{Q}$
given by $v_{s}=T_{0}^{-s_{0}}\circ T_{1}^{-1}\circ T_{0}^{-s_{1}}\circ T_{1}^{-1}\circ\cdots\circ T_{1}^{-1}\circ T_{0}^{-s_{k}}$
a \textit{back tracing function}. If $v_{s}\left(x\right)\in\mathbb{Z}^{+}$
then we say $s$ is an \textit{admissible vector} for $x$, and that
the corresponding back tracing function is a \textit{admissible} for
$x$. Define 
\[
\mathcal{E}\left(x\right)=\left\{ s\in\mathcal{F}:s\ \text{is admissible for }x\right\} .
\]

Wirsching also shows that if $l(s)=m>0$, then there is a unique congruence
class $a$ mod $3^{m}$ with $a$ relatively prime to $3$ such that,
if $x$ is any positive integer, $s$ is an admissible vector for
$x$ if and only if $x\overmod{3^{m}}a$. 

Naturally it would be useful to strengthen the existence theorems
in \cite{Kenny} to determine how a given arithmetic sequence is distributed
in the $3x+1$ graph. More precisely, we wish to determine bounds
for how far away the nearest element in a given arithmetic sequence
$a+d\bbN$ is to a given positive integer $x$. We do so by first
making precise the general bounds that follow from the proof of \cite{Kenny},
Lemma 3.8, and strengthen those bounds for the special case where
$d$ is relatively prime to $2$ and $3$. In every case we show that
the bounds obtained are independent of the choice of $x$, proving
that arithmetic sequences are in this sense uniformly distributed
in the $3x+1$ graph.

\subsection{Back tracing modulo an arbitrary modulus $d$}

We begin with the following bound for the length of a back tracing
sequence to any modulus $d$.
\begin{thm}
Let $d>1$ be a positive integer, and write $d=2^{n}3^{m}b$ where
$b$ is relatively prime to $2$ and $3$. Let $a\in\bbN$ with $a<d$,
and let $f$ be the order of $3/2$ modulo $b$. Then any $x\in\bbN-3\bbN$
back traces to an element of $a+d\bbN$ via an admissible sequence
of length at most $2(b-1)f+n+1$. \end{thm}
\begin{rem}
This bound depends only on the modulus $d$ and not on the starting
position $x$. This shows that the arithmetic sequence $a+d\bbN$
is, in some sense, ``evenly distributed'' throughout the $3x+1$
graph.
\end{rem}
In order to prove this we first prove the case where $n=m=0$, obtaining
a stronger bound in this situation. The construction follows that
of \cite{Kenny}, Lemma 2.8. We sketch the proofs here and refer the
reader to \cite{Kenny} for details.
\begin{lem}
\label{boundmodb} Let $b$ be a positive integer relatively prime
to $2$ and $3$, and let $a\in\{0,1,\ldots,b-1\}$ be any residue
modulo $b$. Let $f$ be the order of $3/2$ modulo $b$. Then for
any positive integer $x$ relatively prime to $3$, there exists a
admissible vector $s\in\mathcal{E}\left(x\right)$ for which $v_{s}(x)\overmod{b}a$
and $v_{s}(x)\underset{3}{\not\equiv}0$, such that 
\begin{equation}
l(s)\le(b-1)f.\label{1}
\end{equation}
 \end{lem}
\begin{proof}
From equation (\ref{eq:x+1}) we have that $T_{0}^{-2}\circ T_{1}\circ T_{0}\circ T_{1}^{-1}\circ T_{0}(n)=n+1$
for any $n$, and so trivially we have that 
\[
T_{0}^{-2}\circ T_{1}\circ T_{0}\circ T_{1}^{-1}\circ T_{0}(n)\overmod{b}n+1.
\]

Let $f$ be the order of $\frac{3}{2}$ modulo $b$, and let $e$
be the order of $2$ modulo $b$. Notice that since $3$ is not congruent
to $2$ mod $b$, $f$ is at least $2$, and similarly $e$ is at
least $2$. We also have $T_{0}=T_{0}^{1-e}$ and $T_{1}=T_{1}^{1-f}$
in $G_{b}$. Substituting, we obtain 
\[
T_{0}^{-2}\circ T_{1}\circ T_{0}\circ T_{1}^{-1}\circ T_{0}=T_{0}^{-2}\circ T_{1}^{1-f}\circ T_{0}^{1-e}\circ T_{1}^{-1}\circ T_{0}^{1-e}.
\]

Let $s_{1}=(2,\underbrace{0,0,\ldots,0}_{f-2},e-1,e-1)$, so that
\[
v_{s_{1}}(n)=T_{0}^{-2}\circ T_{1}^{1-f}\circ T_{0}^{1-e}\circ T_{1}^{-1}\circ T_{0}^{1-e}(n)\overmod{b}n+1.
\]
 Notice that $l(s_{1})=f$.

Now, let $x$ be a positive integer relatively prime to $3$, and
define $s_{2}=(0,0)\cdot\underbrace{s_{1}\cdot s_{1}\cdot\cdots\cdot s_{1}}_{b}$.
Since $2$ is a primitive root mod every power of $3$ (see, for instance,
\cite{Hua}), there is a positive integer $k$ for which $s_{2}\in\mathcal{E}(2^{k}x)$.
Hence $s_{2}\cdot(k)\in\mathcal{E}(x)$. It follows that any vector
of the form $s_{1}\cdot s_{1}\cdot\cdots s_{1}\cdot(k)$, where the
number of copies of $s_{1}$ is at most $b$, is in $\mathcal{E}(x)$
as well.

Let $c=T_{0}^{-k}(x)\bmod b$, and let $s=\underbrace{s_{1}\cdot s_{1}\cdot\cdots\cdot s_{1}}_{(a-c)\bmod b}\cdot(k)$.
Then we have 
\begin{eqnarray*}
v_{s}(x) & \overmod{b} & c+(a-c)\\
 & \overmod{b} & a
\end{eqnarray*}
 and 
\begin{eqnarray*}
l(s) & = & ((a-c)\bmod b)\cdot f\\
 & \le & (b-1)f.
\end{eqnarray*}

Finally, to see that $v_{s}(x)\underset{3}{\not\equiv}0$, let $t=(a-c)\bmod b$
and let 
\[
s_{3}=(0,0)\cdot\underbrace{s_{1}\cdot s_{1}\cdot\cdots\cdot s_{1}}_{b-t}.
\]
 Then $s_{3}\cdot s=s_{2}\cdot(k)$, which is an admissible sequence
for $x$, and so $s_{3}$ is admissible for $v_{s}(x)$. Since $s_{3}$
has length at least $1$, we have that $v_{s}(x)$ is not divisible
by $3$, as desired. 
\end{proof}
With this Lemma in hand, we can now prove Theorem \ref{boundmodb}.
\begin{proof}
First, by Lemma \ref{boundmodb}, we can back trace from $x$ to some
integer $y\underset{3}{\not\equiv}0$ that is congruent to $0$ modulo
$b$ via an admissible sequence of length at most $(b-1)f$. We can
then apply $T_{0}^{-n}$ to $y$ to obtain an integer $z\in\bbN-3\bbN$
that is congruent to $0$ modulo $2^{n}b$.

We wish to back trace from $z$ to an integer $w$ with $w\overmod{2^{n}b}a$
and $w\overmod{3^{m}}a$. Following the arguments in \cite{Kenny},
we can find a sequence $s\in\mathcal{E}\left(z\right)$ of length
at most $(b-1)f+n+1$ for which $v_{s}(z)\overmod{2^{n}b}a$. Since
$2$ is a primitive root mod $3^{l(s)+m}$, there is a power of $2$,
say $2^{k}$, such that $2^{k}z$ also has $s$ as an admissible vector
and moreover $v_{s}(2^{k}z)\overmod{3^{m}}a$ (c.f. \cite{Kenny},
Lemma 3.7). Thus, replacing $s$ by $s\cdot(k)$, we can set $w=v_{s}(z)$,
and we are done.

The total length of the back tracing sequence from $x$ to $w$ is
then at most $(b-1)f+(b-1)f+n+1=2(b-1)f+n+1$, as desired.
\end{proof}

\subsection{Back tracing when $2$ is a primitive root of the modulus}

We can improve this bound in some special cases, particularly when
$2$ is a primitive root modulo $b$. Since the only integers that
can have a primitive root are $2$, $4$, $p^{r}$, and $2p^{r}$
where $p$ is an odd prime, this implies that $b$ must be a power
of an odd prime.
\begin{thm}
\label{backmodb} Let $r$ be a positive integer and $p$ be an odd
prime greater than $3$ such that $2$ is a primitive root modulo
$p^{r}$. Let $a$ be any residue modulo $p^{r}$ relatively prime
to $p$. Then for any positive integer $x$ relatively prime to $3$,
there exists a admissible vector $s\in\mathcal{E}\left(x\right)$
for which $v_{s}(x)\overmod{p^{r}}a$ and $v_{s}(x)\underset{3}{\not\equiv}0$,
such that 
\[
l(s)\le1.
\]
\end{thm}
\begin{proof}
If $x$ is relatively prime to $p$ then since $2$ is a primitive
root, there exists $k$ such that $2^{k}x\xcong{p^{r}}a$. Clearly
$2^{k}x\not\xcong30$ since $x\not\xcong30$. Thus taking $s=\left(k\right)$
gives the desired result.

If $x$ is not relatively prime to $p$ then since $2$ is a primitive
root mod $9$ we can choose $k\geq0$ such that $2^{k+1}x\xcong94$.
So $T_{1}^{-1}\circ T_{0}^{-k}\left(x\right)=\frac{2^{k+1}x-1}{3}$
is an integer that is relatively prime to both $3$ and $p$. Thus
there exists $j$ such that $T_{0}^{-j}\circ T_{1}^{-1}\circ T_{0}^{-k}\left(x\right)=2^{j}\left(\frac{2^{k+1}x-1}{3}\right)\xcong{p^{r}}a$.
Thus taking $s=\left(j,k\right)$ gives the desired result.
\end{proof}
Theorems \ref{boundmodb} and \ref{backmodb} allow us to back trace
to an integer in a desired congruence class mod $b$ that is also
not divisible by $3$, so that we can continue back tracing to obtain
more elements of the same congruence class. Thus, there is an infinite
back tracing sequence $x_{1},x_{2},\ldots$ of elements in $\mathcal{G}$,
satisfying $x_{i}=T(x_{i+1})$ for all $i$, that contains infinitely
many elements congruent to $a$ mod $b$. In section \ref{inverselimits},
we study infinite back tracing sequences in further depth.

\section{Infinite back tracing and inverse limits}

\label{inverselimits}

We first define infinite back tracing sequences in terms of inverse
limits of \textit{level sets} in $\CollatzGraph$, defined as follows.
\begin{defn*}
Let $x$ be a positive integer and let $k$ be a nonnegative integer.
The \textit{$k$th level set of $x$}, which we denote $\mathcal{L}_{k}(x)$,
is the set of all positive integers $y$ for which $T^{k}(y)=x$. \end{defn*}
\begin{rem*}
This is a generalization of the notion of level set defined in \cite{ApplegateLagarias},
which referred only to the level sets of $1$. \end{rem*}
\begin{defn*}
Let $x$ be a positive integer. Consider the directed system $\{\mathcal{L}_{k}(x)\}_{k\ge0}$
where the map from $\mathcal{L}_{k+1}(x)$ to $\mathcal{L}_{k}(x)$
is given by $T$. We define 
\[
\mathcal{I}_{x}=\lim_{\longleftarrow}\mathcal{L}_{k}(x).
\]
We also use the phrase \textit{infinite back tracing sequence from
$x$} to refer to an element of $\mathcal{I}_{x}$, or simply \textit{infinite
back tracing sequence} when $x$ is understood. 
\end{defn*}
Some of the elements of the sets $\mathcal{I}_{x}$ are rather simple
to describe. For instance, recall that when $x$ is divisible by $3$,
one can only ever apply $T_{0}^{-1}$, as the result will never be
congruent to $2$ mod $3$. Thus the only infinite back tracing sequence
from $x=3y$ is $x,2x,4x,8x,\ldots$. For this reason, we primarily
are concerned with the elements of the $3x+1$ graph which are not
divisible by $3$, and we define a modified version of the inverse
limits for the pruned $3x+1$ graph $\widetilde{\mathcal{G}}$, shown
in Figure \ref{Fig:GPruned}.

\begin{figure}[h]
\begin{centering}
\includegraphics{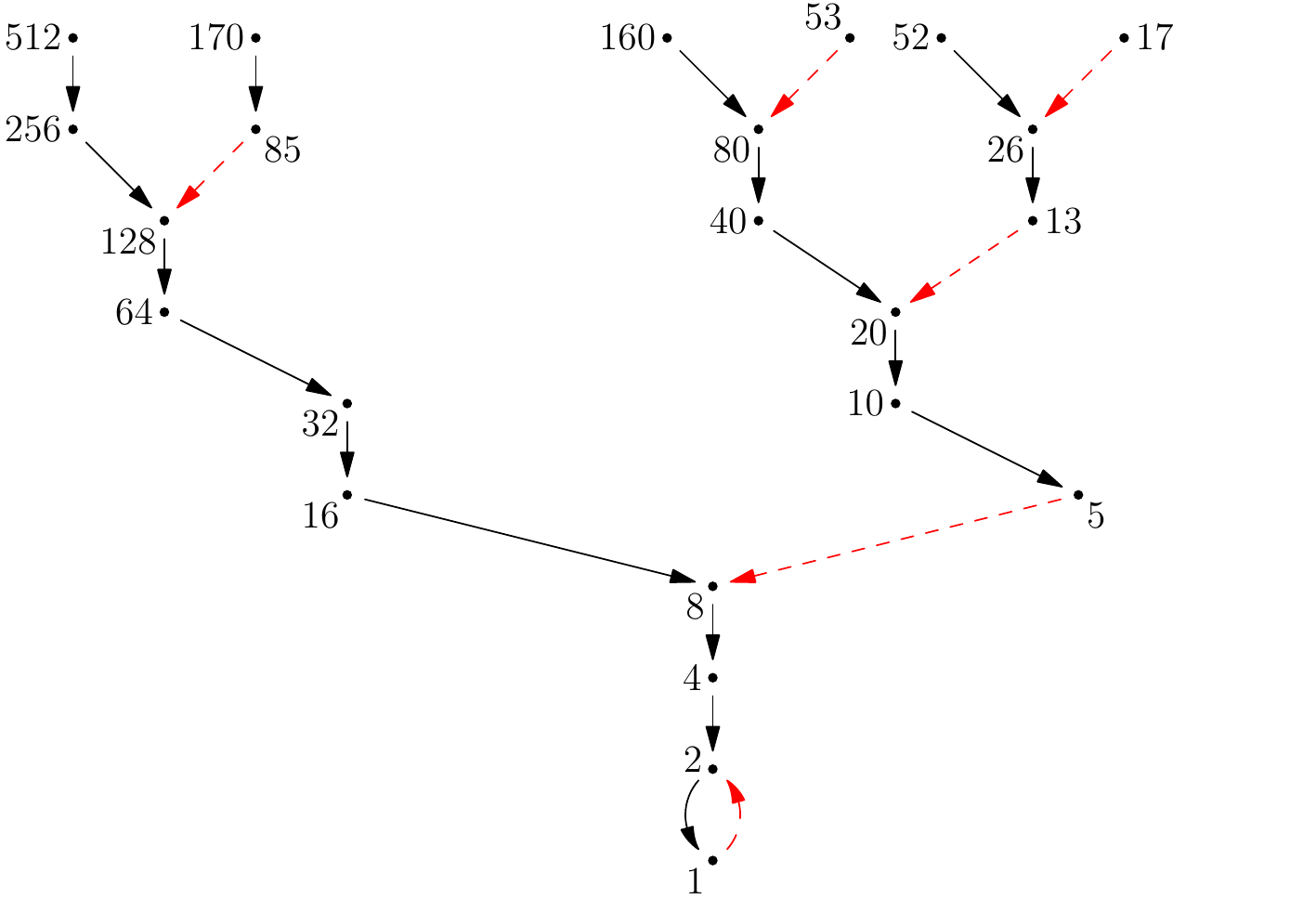} 
\par\end{centering}

\caption{\label{Fig:GPruned} A portion of the pruned $3x+1$ graph $\widetilde{\mathcal{G}}$
near $1$.}
\end{figure}

\begin{defn*}
Let $\widetilde{\mathcal{G}}$ denote the restriction of the $3x+1$
graph to the positive integers relatively prime to $3$, and let $x$
be one such positive integer. Let $\widetilde{\mathcal{L}}_{k}(x)$
be the set of all positive integers $y$ in $\widetilde{\mathcal{G}}$
for which $T^{k}(y)=x$. Then we define 
\[
\widetilde{\mathcal{I}}_{x}=\lim_{\longleftarrow}\widetilde{\mathcal{L}}_{k}(x).
\]

\end{defn*}
Notice that $\widetilde{\mathcal{I}}_{x}$ is strictly contained in
$\mathcal{I}_{x}$ for every $x$.

\subsection{Structure of the inverse limits}

We now investigate the structure of the inverse limit sets $\mathcal{I}_{x}$.
To start, just as the (forward) parity vector of an integer determines
its congruence class mod every power of $2$, we can show that the
parity of the values of an infinite back tracing sequence from $x$
having infinitely many $1$'s determines the congruence class of $x$
mod every power of $3$, and hence determines the integer uniquely.
\begin{defn*}
Let $x$ be a positive integer. A \textit{back tracing parity vector}
(from $x$) is an infinite sequence of $0$'s and $1$'s that is congruent
mod $2$ to some infinite back tracing sequence from $x$.
\end{defn*}
Since we can expand any $\left(s_{0},s_{1},\ldots,s_{n}\right)\in\mathcal{F}$
to the parity vector 
\[
(\underset{s_{0}}{\underbrace{0,\ldots,0}},1,\underset{s_{1}}{\underbrace{0,\ldots,0}},1,\ldots,1,\underset{s_{n}}{\underbrace{0,\ldots,0}},1)
\]
and vice versa, we can say that a finite back tracing parity vector
is admissible for $x$ if and only if the corresponding element of
$\mathcal{F}$ is. An infinite back tracing parity vector is admissible
for $x$ if and only if every initial finite subsequence is.

Notice that for any $3$-adic integer $x$, we can define $T_{0}^{-1}$$\left(x\right)=2x$
and for $x\xcong32$ we can also define $T_{1}^{-1}\left(x\right)=\frac{2x-1}{3}$.
Hence the notion of a back tracing parity vector can be naturally
extended to the $3$-adic integers. Furthermore, for any positive
integer $k$ a $3$-adic integer $\alpha$ is congruent to a unique
ordinary integer $a$ modulo $3^{k}$, and thus a given back tracing
vector is admissible for $\alpha$ if and only if it is admissible
for $a$. 
\begin{thm}
\label{unique} Let $x$ be a $3$-adic integer, and suppose $v$
is a back tracing parity vector for $x$ containing infinitely many
$1$'s. If $v$ is also a back tracing parity vector for the 3-adic
integer $y$, then $x=y$. \end{thm}
\begin{proof}
Let $v_{k}$ be the smallest initial segment of the sequence $v$
containing $k$ $1$'s. Since $v_{k}$ is admissible for both $x$
and $y$, and since there is a unique congruence class modulo $3^{k}$
for which $v_{k}$ is admissible, we must have $x\xcong{3^{k}}y$.
Since $v_{k}$ exists and is finite for every $k$ by assumption,
it follows that $x\xcong{3^{k}}y$ for every $k$ and thus $x=y$.
\end{proof}
Since the positive integers embed naturally in the $3$-adics, we
can easily deduce a similar result for the positive integers for our
purposes.
\begin{cor}
Let $x$ be a positive integer, and suppose $v$ is a back tracing
parity vector for $x$ containing infinitely many $1$'s. If $v$
is also a back tracing parity vector for the positive integer $y$,
then $x=y$. 
\end{cor}
We now study properties of the back tracing vectors themselves. For
the next result, we consider a back tracing parity vector as the binary
expansion of a $2$-adic integer.
\begin{thm}
\label{NotEventuallyPeriodic} Every back tracing parity vector, considered
as a $2$-adic integer, is either: 
\begin{itemize}
\item [\emph{(a)}] a positive integer (i.e., only a finite number of the
digits are nonzero), 
\item [\emph{(b)}] irrational, or 
\item [\emph{(c)}] immediately periodic (i.e., its binary expansion has
the form $\overline{v_{0}\ldots v_{k}}$ where each $v_{i}\in\{0,1\}$). 
\end{itemize}

In particular, if the back tracing parity vector corresponds to an
infinite back tracing sequence in $\mathcal{I}_{1}$, and it is not
the trivial cycle $10101010\ldots$, then it is either an integer
or irrational. 

\end{thm}
\begin{proof}
Let $v$ be a back tracing parity vector for $x$.

It is known that a $2$-adic is a rational number if and only if its
binary expansion is eventually repeating (or immediately repeating).
Thus, if the digits of $v$ are never periodic, then $v$ satisfies
(b).

Now, suppose $v$ is eventually repeating. If its repeating part contains
only $0$'s, it satisfies (a). So, suppose its repeating part contains
at least one $1$. Let $v=v_{0}v_{1}\ldots v_{i}\overline{v_{i+1}v_{i+2}\ldots v_{i+j}}$,
where one of $v_{i+1}\ldots v_{i+j}$ is $1$.

Since $x$ is a positive integer and $v$ is a back tracing parity
vector for $x$, each initial segment of $v$ must correspond to an
admissible back tracing function for $x$. Thus, the value 
\[
x'=T_{v_{0}}^{-1}\circ\cdots\circ T_{v_{i}}^{-1}(x)
\]
 is an integer, and $\overline{v_{i+1}v_{i+2}\ldots v_{i+j}}$ is
a valid back tracing parity vector for $x'$.

Now, let 
\[
x''=T_{v_{i+1}}^{-1}\circ\cdots\circ T_{v_{i+j}}^{-1}(x').
\]
 By a similar argument, $x''$ is an integer and $\overline{v_{i+1}v_{i+2}\ldots v_{i+j}}$
is a valid back tracing parity vector for $x''$. By Theorem \ref{unique},
it follows that $x''=x'$. Thus, we have 
\[
x'=T_{v_{i+1}}^{-1}\circ\cdots\circ T_{v_{i+j}}^{-1}(x'),
\]
 which implies that $T^{j}(x')=x'$. Thus $x'$ is a periodic point
of $T$. But it is impossible to back trace from $x$ into a cycle
of $T$ unless $x$ itself is in the cycle. It follows that $v$ is
in fact immediately periodic, as desired. 
\end{proof}
Notice that, to prove the nontrivial cycles conjecture, it suffices
to show that the only periodic back tracing parity vector for any
positive integer $x$ is the $2$-cycle $\overline{10}$.

The integer back tracing vectors are relatively easy to understand;
they are formed by back tracing a finite number of steps, and then
multiplying by $2$ indefinitely. Occasionally one is forced into
doing so, for $T_{1}^{-1}$ can only be applied to integers congruent
to $2$ mod $3$. If one first back traces to a multiple of $3$,
then multiplying by $2$ will still result in a multiple of $3$,
and one can never apply $T_{1}^{-1}$.

The irrational back tracing parity vectors are not so easy to understand.
As with most irrational numbers, it is difficult to write one down
explicitly. However, we can bound the limiting fraction of $1$'s
in the back tracing parity vectors as follows.
\begin{lem}
Let $v$ be a back tracing parity vector of some positive integer
$x$. Let $k$ be the number of 1's among the first $n$ digits of
$v$ and $p_{n}=k/n$. Then 
\[
\limsup_{n\rightarrow\infty}p_{n}\le\log_{3}(2)\approx0.6309.
\]
\end{lem}
\begin{proof}
Let $f_{n}=T_{v_{0}}^{-1}\circ T_{v_{1}}^{-1}\circ\cdots\circ T_{v_{n-1}}^{-1}$
be the back tracing function corresponding to the first $n$ digits
of $v$. Then $f_{n}(x)$ is a positive integer for all $n$. Thus
there is a minimum value among the values of $f_{n}(x)$. Let $f_{n_{0}}(x)$
be the first occurrence of this minimal value. Then for all $k\ge0$,
$f_{n_{0}}(x)\le f_{n_{0}+k}(x)$.

Now, notice that $T_{1}^{-1}(y)<\frac{2}{3}y$ for all $y$. Therefore,
if a function $f$ is formed by composing $i$ copies of $T_{1}^{-1}$
and $j$ copies of $T_{0}^{-1}$, we have $f(y)\le\left(\frac{2}{3}\right)^{i}2^{j}y$.

Let $t$ be the number of occurrences of $1$ among the first $n_{0}$
digits of $v$, and let $r_{k}$ be the number of occurrences of $1$
among the next $k+1$ digits. Then we have 
\[
f_{n_{0}}(x)\le f_{n_{0}+k}(x)\le\left(2/3\right)^{r}2^{k+1-r}f_{n_{0}}(x)
\]
 and so $3^{r}\le2^{k+1}$. Taking the natural log of both sides,
we find that $r\ln(3)\le(k+1)\ln(2)$. Thus $r\le(k+1)\log_{3}(2)$.

Finally, we have $p_{n_{0}+k}=\frac{r_{k}+t}{n_{0}+k}\le\frac{(k+1)\log_{3}(2)+t}{n_{0}+k}$.
Since $t$ and $n_{0}$ are constant, the right hand side of this
inequality tends to $\log_{3}(2)$ as $k$ approaches infinity, and
so the lim sup of the values of $p_{n}$ is bounded above by this
limit.
\end{proof}

\subsection{Greedy back tracing}

While it is difficult to write down even one irrational infinite back
tracing vector explicitly, there are several ways to obtain such vectors
via a recursion. In particular, we can use a greedy algorithm that
tries to keep the elements of the sequence as small as possible at
each step, with the hopes of gaining insight into the structure of
$\mathcal{G}$ by partitioning it into a union of the following greedy
sequences.
\begin{defn*}
Let $x$ be a positive integer. The \textit{greedy back tracing sequence}
for $x$, denoted $\Greedy(x)$, is the sequence of positive integers
$a_{0},a_{1},\ldots$ defined recursively by $a_{0}=x$ and for all
$i>0$ 
\[
a_{i+1}=\begin{cases}
T_{1}^{-1}\left(a_{i}\right) & \mbox{if }T_{1}^{-1}\left(a_{i}\right)\not\xcong30\\
T_{0}^{-1}\left(a_{i}\right) & \mbox{otherwise}
\end{cases}
\]
We also write $V_{x}$ to denote $\Greedy(x)$ taken mod $2$, the
back tracing parity vector of $\Greedy(x)$.
\end{defn*}
It is easily verified that the recursion for $\Greedy(x)$ can also
be written as $a_{0}=x$ and for all $i>0$ 
\[
a_{i+1}=\begin{cases}
T_{1}^{-1}\left(a_{i}\right) & \mbox{if }a_{i}\xcong92\mbox{ or }a_{i}\xcong98\\
T_{0}^{-1}\left(a_{i}\right) & \mbox{otherwise}
\end{cases}
\]

We now show that for $x$ relatively prime to $3$, the back tracing
parity vector $V_{x}$ corresponding to $\Greedy(x)$ has infinitely
many $1$'s, and therefore that, for instance, $V_{4}$ is irrational.
\begin{lem}
\label{3zeros} Let $x$ be a positive integer relatively prime to
$3$. Then $V_{x}$ can have at most three $0$'s in a row at any
point in the sequence. \end{lem}
\begin{proof}
It suffices to show if $y$ is an odd positive integer relatively
prime to $3$, the greedy algorithm applies $T_{0}^{-1}$ at most
three times before applying a $T_{1}^{-1}$.

Suppose $y$ is an odd positive integer relatively prime to $3$.
Then it is congruent to one of $1$, $2$, $4$, $5$, $7$, or $8$
mod $9$.

\textit{Case 1.} Suppose $y$ is congruent to $2$ or $8$ mod $9$.
Then the greedy algorithm applies $T_{1}^{-1}$, and we are done.

\textit{Case 2.} Suppose $y$ is congruent to $1$ or $4$ mod $9$.
Then the greedy algorithm determines that the next integer in the
sequence is $T_{0}^{-1}(y)=2y$, which is congruent to $2$ or $8$
mod $9$. At this point, $T_{1}^{-1}$ is applied, and we are done.

\textit{Case 3.} Suppose $y$ is congruent to $5$ mod $9$. Then
the greedy algorithm applies $T_{0}^{-1}$ to yield an integer congruent
to $1$ mod $9$. Then, $T_{0}^{-1}$ is applied again to obtain an
integer congruent to $2$ mod $9$, and $T_{1}^{-1}$ is applied.

\textit{Case 4.} Suppose $y$ is congruent to $7$ mod $9$. Then
the greedy algorithm applies $T_{0}^{-1}$ to yield an integer congruent
to $5$ mod $9$, and by the above argument, two more $T_{0}^{-1}$'s
are used before applying $T_{1}^{-1}$. 
\end{proof}
We immediately obtain the following fact about greedy vectors.
\begin{cor}
Let $x$ be a positive integer, let $k$ be the number of 1's among
the first $n$ digits of $v$ and $p_{n}=k/n$. Then 
\[
\liminf_{n\rightarrow\infty}p_{n}\ge\frac{1}{4}.
\]
\end{cor}
\begin{proof}
If the $n$th term of $V_{x}$ is $1$, then each $1$ in the first
$n$ terms is preceded by no more than three $0$'s, by Lemma \ref{3zeros}.
It follows that $p_{n}\ge1/4$ in this case.

Otherwise, the $n$th term is $0$, and the first $n$ terms end in
a string of $k$ zeroes, where $1\le k\le3$. The first $n-k$ terms,
however, have the property that each $1$ is preceded by at most three
$0$'s, so there are at least $(n-k)/4\ge\frac{n-3}{4}$ ones among
the first $n$ terms. As $n$ approaches infinity, the lower bound
approaches $n/4$, and so $\liminf_{n\rightarrow\infty}p_{n}\ge\frac{1}{4}$,
as desired. 
\end{proof}
This gives a lower bound on the limiting percentage of $1$'s in a
greedy back tracing vector. Since we are greedily choosing to apply
$T_{1}^{-1}$ whenever possible, it would be of interest to determine
whether the greedy algorithm does maximize the percentage of $1$'s
in a back tracing parity vector starting from $x$, and what that
percentage is precisely. We leave this as an open problem for further
study.

Having studied infinite back tracing sequences in some depth, we return
to the problem of finding arithmetic progressions in the $3x+1$ graph.

\subsection{A strongly sufficient arithmetic progression}

We can obtain surprising information about infinite back tracing sequences
when we look modulo certain integers. We begin by proving the following
remarkable fact.
\begin{thm}
\label{2mod9} Let $x$ be a positive integer relatively prime to
$3$. Then every infinite back tracing sequence in $\widetilde{\mathcal{I}}_{x}$
contains a positive integer congruent to $2$ mod $9$. \end{thm}
\begin{proof}
We first draw a directed graph to represent of the action of $T_{0}$
and $T_{1}$ on the elements of $\mathbb{Z}/9\mathbb{Z}$ relatively
prime to $3$, as shown in figure \ref{9}. Denote this directed graph
by $\widetilde{\Gamma}_{9}$. Notice that any infinite back tracing
sequence in $\widetilde{\mathcal{I}}_{x}$, taken mod $9$, defines
a sequence of residues traced out by an infinite path along the arrows
in $\widetilde{\Gamma}_{9}$ in the \textit{reverse} direction (against
the arrows). We call such a path a \textit{reverse path}.

\begin{figure}[h]
\begin{centering}
\includegraphics{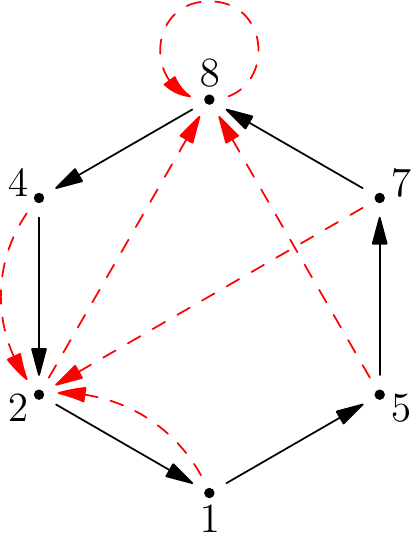} 
\par\end{centering}

\caption{\label{9}The action of $T_{0}$ and $T_{1}$ on the residues mod
$9$ relatively prime to $3$.}
\end{figure}

Let $v\in\widetilde{\mathcal{I}}_{x}$ be an arbitrary back tracing
sequence avoiding multiples of $3$, and let $P$ be the corresponding
reverse path on $\Gamma_{9}$. Assume to the contrary that $v$ is
does not contain an integer congruent to $2$ mod $9$. Then the path
$P$ avoids the node labeled $2$, and so it lies entirely in the
subgraph $\Gamma_{9}'$ shown in Figure \ref{9no2}.

\begin{figure}[h]
\begin{centering}
\includegraphics{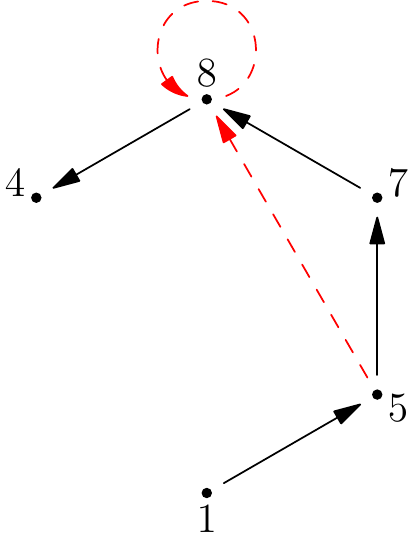} 
\par\end{centering}

\caption{\label{9no2}The subgraph $\Gamma_{9}'$ of $\widetilde{\Gamma}_{9}$
formed by deleting the node labeled $2$.}
\end{figure}

Now, since the path $P$ in is infinite, it cannot contain the nodes
$7$, $5$, or $1$, since if we travel backwards along the edges
from these nodes we must end up at $1$, from which we cannot travel
further. Furthermore, if $P$ begins at the node $4$, then it must
travel to the node $8$, where it is locked into the red loop at $8$.
But by Theorem \ref{NotEventuallyPeriodic}, this is impossible. Thus
$P$ is the cyclic path $8,8,8,\ldots$.

Since $P$ must consist only of red arrows, the back tracing parity
vector corresponding to $v$ is the all-$1$'s vector, which is a
valid back tracing parity vector for the integer $-1$. But by Theorem
\ref{unique}, it can therefore not be a valid back tracing parity
vector for any other $3$-adic integer, and in particular, it cannot
be a valid back tracing parity vector for any positive integer. Thus,
we have a contradiction, and so $v$ must in fact contain an integer
congruent to $2$ mod $9$. 
\end{proof}
Thus, for the arithmetic sequence $S=\{2+9n\}$, not only can we back
trace from any $x\not\overmod{3}0$ to an element of $S$, but we
cannot avoid doing so no matter how we back trace from $x$. We say
such sets $S$ are \textit{strongly sufficient in the backward direction},
or simply \textit{backward sufficient}.

Notice that the same argument applies to the forward direction: by
looking at (ordinary, not reverse) paths in $\widetilde{\Gamma}_{9}$,
we see that any forward $T$-orbit must contain an integer congruent
to $2$ mod $9$.
\begin{cor}
The $T$-orbit of every positive integer contains an integer congruent
to $2$ mod $9$.
\end{cor}
This essentially ``proves the Collatz conjecture mod $9$'', and
we say that $S=\{2+9n\}$ is \textit{strongly sufficient in the forward
direction}, or simply \textit{forward sufficient}. We define these
notions precisely in the next section.

\section{Strong sufficiency and directed graphs}

\label{forward}

We define strong sufficiency in both the forward and backward directions,
and also for the special case of nontrivial cycles, as follows.
\begin{defn*}
Let $S$ be a set of positive integers. Then 
\begin{itemize}
\item $S$ is \textit{forward sufficient} if every divergent $T$-orbit
contains an element of $S$. 
\item $S$ is \textit{cycle sufficient} if every nontrivial cycle contains
an element of $S$. 
\item $S$ is \textit{backward sufficient} if for every positive integer
$x$ relatively prime to $3$, every element of $\widetilde{\mathcal{I}}_{x}$
having an irrational back tracing parity vector contains an element
of $S$. 
\item $S$ is \textit{strongly sufficient} if it is forward sufficient,
cycle sufficient, and backward sufficient. 
\end{itemize}
\end{defn*}
\begin{notation*}
For simplicity in what follows, we write $\{a_{1},\ldots,a_{k}\bmod d\}$
to denote the set of positive integers congruent to one of $a_{1},a_{2},\ldots,a_{k}$
mod $d$. We sometimes drop the brackets when the notation is clear.
So we would say that $\{2\bmod9\}$, or $2\bmod9$, is strongly sufficient.\end{notation*}
\begin{rem*}
In the case of $2$ mod $9$, we did not need to restrict our claim
to the divergent orbits, the nontrivial cycles, and the aperiodic
infinite back tracing sequences, because the cycle $\overline{1,2}$
itself contains an element congruent to $2$ mod $9$. However, there
are many sets $S$ that intersect those $T$-orbits and back tracing
sequences that do not end in $\overline{1,2}$, but do not intersect
\textit{every} $T$-orbit simply because $S$ does not contain $1$
or $2$. For this reason, we throw away the back tracing sequences
and orbits that end in $\overline{1,2}$ in our definition of strong
sufficiency.
\end{rem*}
Notice also that it suffices to prove the $3x+1$ conjecture for the
elements of any single strongly sufficient set, and that for any strongly
sufficient set $S$, $S\cup\{1\}$ and $S\cup\{2\}$ are sufficient
sets. Moreover, not only does every positive integer $x$ \textit{merge}
with an element of $S\cup\{1\}$ (or $S\cup\{2\}$), but it actually
\textit{contains} one in its $T$-orbit and in every element of $\widetilde{\mathcal{I}}_{x}$.
Hence the term ``strong sufficiency.'' 
\begin{rem*}
Obtaining strongly sufficient sets give us a promising way in which
to approach the nontrivial cycles conjecture. In particular, suppose
we can show that a fixed, finite set of residues $a_{1},\ldots,a_{k}$,
modulo a set of arbitrarily large values of $n$, is strongly sufficient
for each of these moduli $n$. Then any nontrivial cycle, being bounded,
must contain one of the positive integers $a_{1},\ldots,a_{k}$, and
so we would only need to verify that the finite list of positive integers
$a_{1},\ldots,a_{k}$ have a $T$-orbit that contains $1$. 
\end{rem*}
In light of this remark, we begin a search for strongly sufficient
sets. To do so, we first define a generalization of the directed graph
$\Gamma_{9}$.

\subsection{The graphs $\Gamma_{d}$ and $\widetilde{\Gamma}_{d}$}

We define the graph $\Gamma_{d}$ to be the $3x+1$ graph $\mathcal{G}$
taken modulo $d$, as follows.
\begin{defn*}
For a positive integer $k$, define $\Gamma_{k}$ to be the two-colored
directed graph on $\mathbb{Z}/k\mathbb{Z}$ such that 
\begin{itemize}
\item there is a black arrow from $r$ to $s$ if and only if there exist
positive integers $x$ and $y$ with $x\overmod{k}r$ and $y\overmod{k}s$
with $T_{0}(x)=y$, and 
\item there is a red arrow from $r$ to $s$ if and only if there exist
positive integers $x$ and $y$ with $x\overmod{k}r$ and $y\overmod{k}s$
with $T_{1}(x)=y$. 
\end{itemize}
\end{defn*}
As with the $3x+1$ graph $\mathcal{G}$, since we are primarily interested
in the portions of $T$-orbits and infinite back tracing sequences
whose elements are all relatively prime to $3$, we also consider
the pruned graph $\widetilde{\mathcal{G}}$ taken modulo $d$.
\begin{defn*}
For $d\overmod{3}0$, the pruned graph $\widetilde{\Gamma}_{d}$ is
the subgraph of $\Gamma_{d}$ formed by deleting the nodes divisible
by $3$ (along with all of their adjacent edges). When $d\overmod{3}0$,
we define $\widetilde{\Gamma}_{d}=\Gamma_{d}$. 
\end{defn*}
Notice that when we refer to a node $z$ in $\widetilde{\Gamma}_{d}$
or $\Gamma_{d}$ we identify the congruence class $z$ with the integer
in $\left\{ 0,1,\ldots,d-1\right\} $ that is in that class. For $b$
relatively prime to $2$ and $3$, the graph $\Gamma_{b}$ is a natural
representation of the action of $T_{0}$ and $T_{1}$ on $\mathbb{Z}/b\mathbb{Z}$
in the group $G_{b}$. Examples are given in Figure \ref{exampledigraphs}.

\begin{figure}[h]
\begin{centering}
\includegraphics{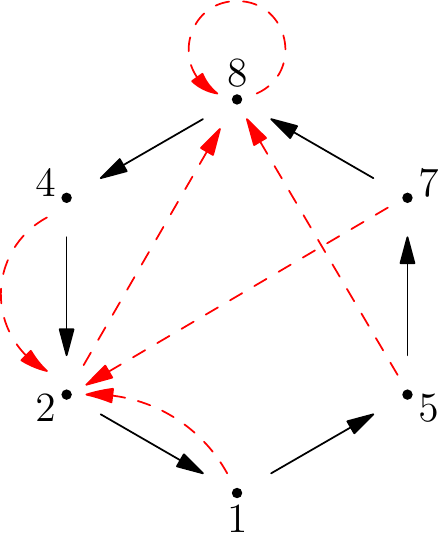} 
\par\end{centering}

\caption{\label{exampledigraphs} The digraphs $\widetilde{\Gamma}_{7}=\Gamma_{7}$
(top left), $\widetilde{\Gamma}_{8}=\Gamma_{8}$ (bottom), and $\Gamma_{9}$
(top right), which strictly contains $\widetilde{\Gamma}_{9}$.}
\end{figure}

We now demonstrate several basic properties of the graphs $\Gamma_{d}$
for various $d$.
\begin{prop}
\label{GammaStructure}Let $d$ be a positive integer. 
\begin{itemize}
\item [\emph{(a)}]\label{part1} If $d$ is even, each even node has two
black arrows and no red arrows coming from it, each odd node has two
red arrows and no black arrows coming from it. If $d$ is odd, each
node has exactly one red and black arrow coming from it. 
\item [\emph{(b)}]\label{part2} If $d\overmod{3}0$, each node congruent
to $2$ modulo $3$ has exactly one black arrow and at least red arrow
pointing to it, and each other node has one black arrow and no red
arrows pointing to it. If $d\underset{3}{\not\equiv}0$, then every
node has one black and one red arrow pointing to it. 
\item [\emph{(c)}]\label{part3} If $d$ is relatively prime to $2$ and
$3$, then in $\Gamma_{d}$, the black arrows form disjoint cycles
on the vertices, as do the red arrows. There is one black loop at
$0$, and each of the other black cycles have length dividing the
order of $2$ mod $d$. There is one red loop at $d-1$, and each
of the other red cycles have length dividing the order of $3/2$ mod
$d$. 
\item [\emph{(d)}]\label{part4} If $d=3^{m}$ for some $m$, then in
$\Gamma_{d}$, the black arrows form a single cycle on the nodes which
are relatively prime to $3$. For $i=1,\ldots,m-1$, there is also
a cycle of black arrows consisting of the nodes divisible by $3^{i}$
but not by $3^{i+1}$, and a black loop at the node $0$. 
\item [\emph{(e)}]\label{part5} If $d=3^{m}$ for some $m$, then in
$\Gamma_{d}$, the red arrows form a rooted oriented tree with $3^{m}-1$
as the root and all arrows oriented towards the root, plus a red loop
at the root. The length of the shortest red path from any leaf to
the root is $m$. 
\end{itemize}
\end{prop}
\begin{proof}
Claim (c) follows from Lemma \ref{orders} and the fact that $T_{0}$
and $T_{1}$ generate a permutation group on $\mathbb{Z}/b\mathbb{Z}$.

For claim (a), note that dividing by $2$ modulo some even $d$ can
be done in two ways: either $2r\mapsto r$ or $2r\mapsto r+\frac{1}{2}d$.
Thus, if we send a congruence class $x$ to $x/2$ or to $(3x+1)/2$
modulo $d$, in both cases we have exactly two possible results for
the congruence class of $T(x)$ mod $2^{n}$. It follows that each
even node has two red arrows and two black arrows coming from it.
If $d$ is odd, then $2$ is invertible modulo $d$, and so $T(x)$
is well defined.

For claim (b), note that $T_{0}^{-1}(x)=2x$ is a well-defined function
on $\mathbb{Z}/d\mathbb{Z}$ for all $d$, but $T_{1}^{-1}(x)=(2x-1)/3$
is well-defined if and only if $d$ is not divisible by $3$. Thus,
if $d$ is not divisible by $3$, there is one red and one black arrow
pointing to every node. If $d$ is divisible by $3$, however, then
only those $x$ congruent to $2$ modulo $3$ can have a red arrow
pointing to it.

For part (d), it is known that $2$ is a primitive root mod $3^{m}$
(see \cite{Hua}), so the black arrows behave as described.

We now prove part (e). To do so, we consider the purely red back tracing
paths starting at integers congruent to $-1$ mod $3^{m}$. Suppose
$x\overmod{3^{m}}-1$, and let $M\ge m$ be the largest positive integer
such that $x\overmod{3^{M}}-1$. Then we can write $x=3^{M}k-1$ where
$k$ is relatively prime to $3$.

Back tracing along a red arrow from $x$, we have that $T_{1}^{-1}(x)=(2x-1)/3=2\cdot3^{M-1}k-1$,
so $T_{1}^{-1}(x)$ is congruent to $-1$ mod $M-1$. If $M-1\ge m$,
we have that it is also congruent to $-1$ mod $m$. Thus, for the
first $M-m$ steps in back tracing along red arrows, we follow a self-loop
in $\Gamma_{3^{m}}$ from $3^{m}-1$ to itself. In particular, this
loop exists in the graph, since there are positive integers congruent
to $-1$ mod $3^{M}$ for any $M>m$.

Now, choose an integer $x$ such that $M=m$, that is, $m$ is the
maximum positive integer for which $x\overmod{3^{m}}-1$. Then by
a similar argument, $T_{1}^{-1}(x)\overmod{3^{m-1}}-1$, and by induction
we have 
\[
(T_{1}^{-1})^{k}(x)\overmod{3^{m-k}}-1
\]
 for all $k\ge0$. Thus $(T_{1}^{-1})^{m-1}(x)\overmod{3}-1$. It
follows that $(T_{1}^{-1})^{m}(x)$ is congruent to either $0$ or
$1$ mod $3$, and so we cannot back trace using $T_{1}^{-1}$ any
further from here.

Note also that for each step in this process, the maximum $M$ for
which $(T_{1}^{-1})^{k}(x)\overmod{3^{M}}-1$ is monotone decreasing
by $1$ at each step. Thus we can partition the congruence classes
mod $3^{m}$ into grades based on the value of $M$, with the final
grade consisting of those residues congruent to $0$ or $1$ mod $3$,
and we see that each element in the back tracing sequence from $x$
is in a distinct grade. Moreover, each of these sequences, starting
from $M=m$, has length $m$. It follows that the red arrows do indeed
form a tree oriented towards the root at $3^{m}-1$, with the shortest
path from any leaf to the root having length $m$. 
\end{proof}

\subsection{Vertex minors and strong sufficiency}

Using the digraphs $\Gamma_{d}$ for various $d$, we can obtain several
new strongly sufficient sets. In order to do so, we first give a graph-theoretic
criterion for strong sufficiency. 
\begin{prop}
\label{strongsufficiency} Let $d\in\mathbb{N}$, and let $a_{1},\ldots,a_{k}$
be $k$ distinct residues mod $d$. Define $\Gamma_{d}'$ to be the
subgraph of $\widetilde{\Gamma}_{d}$ formed by deleting the nodes
labeled $a_{1},\ldots,a_{k}$ and all arrows connected to them and
define $\Gamma_{d}''$ to be the graph formed by deleting any edge
from $\Gamma_{d}'$ that is not contained in any cycle in $\Gamma_{d}'$.
If $\Gamma_{d}''$ is a disjoint union of cycles and isolated vertices,
and each of the cycles have length less than $630,\!138,\!897$, then
the set 
\[
a_{1},\ldots,a_{k}\bmod d
\]
 is strongly sufficient. \end{prop}
\begin{proof}
Suppose $\Gamma_{d}''$ is a disjoint union of cycles and isolated
vertices, and each of the cycles have length less than $630,\!138,\!897$.

We first show that $\{a_{1},\ldots,a_{k}\bmod d\}$ is forward sufficient
and cycle sufficient. Assume for contradiction that there is a positive
integer $x$ whose $T$-orbit, taken mod $d$, does \textit{not} end
in the cycle $\overline{1,2}$ and also avoids the set $\{a_{1},\ldots,a_{k}\bmod d\}$.
Since we are interested in the long-term behavior of this orbit, we
may assume without loss of generality that $x$ and all the elements
of its $T$-orbit are relatively prime to $3$, and hence trace out
an infinite path $P$ in $\widetilde{\Gamma}_{d}$. Since the $T$-orbit
of $x$ avoids $\{a_{1},\ldots,a_{k}\bmod d\}$, it follows that $P$
lies entirely in $\Gamma_{d}'$.

Now, for any edge $e$ in $\Gamma_{d}'$ that is not contained in
any cycle, the path $P$ contains $e$ at most once. Thus, some infinite
tail of the path $P$ does not contain $e$, and so there is a $T$-orbit
of some positive integer whose corresponding path does not contain
$e$. We can thus assume without loss of generality that $P$ does
not pass through $e$.

Using the same argument on each such edge $e$, we can assume that
$P$ lies on the subgraph $\Gamma_{d}''$ formed by deleting these
edges. Since $P$ is an infinite path, it must be contained in one
of the loops of $\Gamma_{d}''$. Thus $P$ is periodic. Its parity
vector is also periodic, determined by the color of the edges on the
loop, and so the $T$-orbit of $x$ is periodic, corresponding to
a nontrivial cycle with period equal to the length of the loop. But
by our assumptions, the length of the loop is less than $630,\!138,\!897$,
and it is not the cycle $\overline{1,2}$. But there are no such positive
integer cycles (see \cite{Sinisalo}), and so we have a contradiction.

For strong sufficiency in the backward direction, the same argument
can be applied to the graph formed by reversing the arrows in $\widetilde{\Gamma}_{d}$,
and hence in $\Gamma_{d}'$. \end{proof}
\begin{rem*}
If we remove the bound $630,\!138,\!897$ on the length of the loops,
the criterion shows that the set is forward and backward sufficient,
but not necessarily cycle sufficient.
\end{rem*}
Using Proposition \ref{strongsufficiency}, we have obtained the list
of strongly sufficient sets shown in Table \ref{loopsvalues}.

\afterpage{\clearpage{}}

\begin{table}
\begin{centering}
\begin{tabular}{rrrrrr}
\toprule 
\multicolumn{6}{c}{\textbf{Strongly sufficient sets}}\tabularnewline
\midrule 
$0\bmod2$  & $1,4\bmod9$  & $1,2,6\bmod7$  & $3,4,7\bmod10$  & $2,7,8\bmod11$  & $4,5,12\bmod14$ \tabularnewline
\midrule 
$1\bmod2$  & $1,8\bmod9$  & $0,1,3\bmod8$  & $3,6,7\bmod10$  & $3,4,5\bmod11$  & $4,6,11\bmod14$ \tabularnewline
\midrule 
$1\bmod3$  & $4,5\bmod9$  & $0,1,6\bmod8$  & $3,7,8\bmod10$  & $3,4,8\bmod11$  & $4,11,12\bmod14$ \tabularnewline
\midrule 
$2\bmod3$  & $4,7\bmod9$  & $2,4,7\bmod8$  & $4,5,7\bmod10$  & $3,4,9\bmod11$  & $6,7,8\bmod14$ \tabularnewline
\midrule 
$1\bmod4$  & $5,8\bmod9$  & $2,5,7\bmod8$  & $5,6,7\bmod10$  & $3,4,10\bmod11$  & $6,8,9\bmod14$ \tabularnewline
\midrule 
$2\bmod4$  & $7,8\bmod9$  & $0,1,4\bmod10$  & $5,7,8\bmod10$  & $3,6,10\bmod11$  & $7,8,12\bmod14$ \tabularnewline
\midrule 
$2\bmod6$  & $4,7\bmod11$  & $0,1,6\bmod10$  & $0,1,5\bmod11$  & $1,7,10\bmod12$  & $8,9,12\bmod14$ \tabularnewline
\midrule 
$2\bmod9$  & $5,6\bmod11$  & $0,1,8\bmod10$  & $0,1,8\bmod11$  & $1,8,11\bmod12$  & $1,5,7\bmod15$ \tabularnewline
\midrule 
$0,3\bmod4$  & $6,8\bmod11$  & $0,2,4\bmod10$  & $0,1,9\bmod11$  & $2,4,11\bmod12$  & $1,5,11\bmod15$ \tabularnewline
\midrule 
$0,1\bmod5$  & $6,9\bmod11$  & $0,2,6\bmod10$  & $0,2,5\bmod11$  & $4,7,10\bmod12$  & $1,5,13\bmod15$ \tabularnewline
\midrule 
$0,2\bmod5$  & $1,5\bmod12$  & $0,2,7\bmod10$  & $0,2,8\bmod11$  & $1,3,4\bmod13$  & $1,5,14\bmod15$ \tabularnewline
\midrule 
$1,3\bmod5$  & $2,5\bmod12$  & $0,2,8\bmod10$  & $0,4,5\bmod11$  & $1,4,6\bmod13$  & $1,7,8\bmod15$ \tabularnewline
\midrule 
$2,3\bmod5$  & $2,8\bmod12$  & $0,4,7\bmod10$  & $0,4,8\bmod11$  & $1,8,11\bmod13$  & $1,8,13\bmod15$ \tabularnewline
\midrule 
$1,4\bmod6$  & $2,10\bmod12$  & $0,6,7\bmod10$  & $0,4,9\bmod11$  & $2,3,7\bmod13$  & $1,8,14\bmod15$ \tabularnewline
\midrule 
$1,5\bmod6$  & $4,5\bmod12$  & $0,7,8\bmod10$  & $1,2,7\bmod11$  & $2,6,7\bmod13$  & $1,10,11\bmod15$ \tabularnewline
\midrule 
$4,5\bmod6$  & $5,8\bmod12$  & $1,3,4\bmod10$  & $1,3,5\bmod11$  & $3,4,9\bmod13$  & $1,10,13\bmod15$ \tabularnewline
\midrule 
$2,3\bmod7$  & $7,8\bmod12$  & $1,3,6\bmod10$  & $1,3,8\bmod11$  & $3,4,10\bmod13$  & $2,5,7\bmod15$ \tabularnewline
\midrule 
$2,5\bmod7$  & $8,11\bmod15$  & $1,3,8\bmod10$  & $1,3,9\bmod11$  & $3,7,10\bmod13$  & $2,5,11\bmod15$ \tabularnewline
\midrule 
$3,4\bmod7$  & $1,8\bmod18$  & $1,4,5\bmod10$  & $1,3,10\bmod11$  & $3,10,11\bmod13$  & $2,5,13\bmod15$ \tabularnewline
\midrule 
$4,5\bmod7$  & $2,8\bmod18$  & $1,5,6\bmod10$  & $1,5,7\bmod11$  & $4,6,9\bmod13$  & $2,5,14\bmod15$ \tabularnewline
\midrule 
$4,6\bmod7$  & $2,11\bmod18$  & $1,5,8\bmod10$  & $1,7,8\bmod11$  & $4,6,10\bmod13$  & $2,7,8\bmod15$ \tabularnewline
\midrule 
$1,4\bmod8$  & $7,8\bmod18$  & $2,3,4\bmod10$  & $1,7,9\bmod11$  & $4,8,9\bmod13$  & $2,7,10\bmod15$ \tabularnewline
\midrule 
$1,5\bmod8$  & $8,10\bmod18$  & $2,3,6\bmod10$  & $2,3,5\bmod11$  & $6,7,10\bmod13$  & $2,8,13\bmod15$ \tabularnewline
\midrule 
$2,3\bmod8$  & $8,14\bmod18$  & $2,3,7\bmod10$  & $2,3,7\bmod11$  & $6,10,11\bmod13$  & $2,8,14\bmod15$ \tabularnewline
\midrule 
$2,6\bmod8$  & $10,11\bmod18$  & $2,3,8\bmod10$  & $2,3,8\bmod11$  & $7,8,9\bmod13$  & $2,10,11\bmod15$ \tabularnewline
\midrule 
$3,4\bmod8$  & $5,11\bmod21$  & $2,4,5\bmod10$  & $2,3,9\bmod11$  & $8,9,11\bmod13$  & $2,10,13\bmod15$ \tabularnewline
\midrule 
$3,5\bmod8$  & $0,1,3\bmod7$  & $2,5,6\bmod10$  & $2,3,10\bmod11$  & $8,10,11\bmod13$  & $2,10,14\bmod15$ \tabularnewline
\midrule 
$4,6\bmod8$  & $0,1,5\bmod7$  & $2,5,7\bmod10$  & $2,5,7\bmod11$  & $3,4,10\bmod14$  & $4,5,11\bmod15$ \tabularnewline
\midrule 
$5,6\bmod8$  & $0,1,6\bmod7$  & $2,5,8\bmod10$  & $2,6,7\bmod11$  & $4,5,6\bmod14$  & $4,10,11\bmod15$ \tabularnewline
\bottomrule
\end{tabular}
\par\end{centering}

\caption{\label{loopsvalues} Some strongly sufficient sets. Each entry reveals
a new property of the divergent $T$-orbits and nontrivial cycles.
For instance, every divergent $T$-orbit, nontrivial cycle, and aperiodic
infinite back tracing sequence in $\widetilde{\mathcal{G}}$ contains
an element congruent to either $5$ or $11$ mod $21$.}
\end{table}

\subsection{Forward, backward, and cycle sufficiency}

The sets in Table \ref{loopsvalues} are all strongly sufficient.
In this section, we use more powerful tools to obtain sets that are
not necessarily strongly sufficient, but are strongly sufficient in
the forward or backward direction or cycle sufficient.

We require some known results on the limiting percentage of odd numbers
in a $T$-orbit. In \cite{Eliahou}, Eliahou showed that if a $T$-cycle
of positive integers of length $n$ contains $r$ odd positive integers
(and $n-r$ even positive integers), and has minimal element $m$
and maximal element $M$, then 
\[
\frac{\ln(2)}{\ln\left(3+\frac{1}{m}\right)}\le\frac{r}{n}\le\frac{\ln(2)}{\ln\left(3+\frac{1}{M}\right)}.
\]

In \cite{Lagarias}, Lagarias showed a similar result for divergent
orbits: the percentage of odd numbers in any divergent orbit is \textit{at
least} $\ln(2)/\ln(3)\approx.6309$.

We also require a similar bound for infinite back tracing sequences.
Let $x$ be a positive integer relatively prime to $3$ and let $x=x_{0},x_{1},x_{2},\ldots$
be an infinite back tracing sequence in $\widetilde{\mathcal{I}}_{x}$.
Suppose further that the sequence is not periodic. Then by Proposition
\ref{NotEventuallyPeriodic}, its back tracing parity vector is either
irrational or has only finitely many $1$'s.

In the case that the parity vector is irrational, note that every
positive integer occurs at most a finite number of times in the sequence
(otherwise, the sequence must be a cycle containing that integer).
In particular, there is some $N$ such that for all $n>N$, $x_{n}>x$.
Now, consider the function $f$ defined by the composition of the
first $n$ applications of $T_{0}^{-1}$ or $T_{1}^{-1}$ in this
back tracing sequence. Then since $T_{1}^{-1}(y)\le2y/3$ for any
positive integer $y$, we have that $f(x)\le\left(\frac{2}{3}\right)^{r}\cdot2^{n-r}x$
where $r$ is the number of $1$'s among the first $n$ digits of
the back tracing parity vector. It follows that 
\[
x<x_{n}=f(x)=\left(\frac{2}{3}\right)^{r}\cdot2^{n-r}x
\]
 and therefore 
\[
1<\left(\frac{2}{3}\right)^{r}\cdot2^{n-r}.
\]
 Taking the natural log of both sides and solving for $r/n$, we obtain
\[
\frac{r}{n}\le\frac{\ln2}{\ln3}.
\]

In the case that the parity vector has only finitely many $1$'s,
there is clearly an $N$ for which the same inequality holds for all
$n>N$. Thus, the percentage of odd numbers in any aperiodic infinite
back tracing sequence is \textit{at most} $\ln(2)/\ln(3)\approx.6309$.

We summarize these results in the following proposition.
\begin{prop}
Let $\rho=\ln(2)/\ln(3)\approx.6309$.\label{percentage}
\begin{itemize}
\item [\emph{(a)}]The percentage of odd numbers in any divergent orbit
is at least $\rho$. 
\item [\emph{(b)}]The percentage of odd numbers in any nontrivial cycle
with minimal element $m$ and maximal element $M$ is bounded below
by $\frac{\ln(2)}{\ln\left(3+\frac{1}{m}\right)}$ and above by $\frac{\ln(2)}{\ln\left(3+\frac{1}{M}\right)}$. 
\item [\emph{(c)}]The percentage of odd numbers in any aperiodic infinite
back-tracing sequence is at most $\rho$. 
\end{itemize}
\end{prop}
Using this as a tool, we obtain the following result.
\begin{thm}
The arithmetic sequence $\{20\bmod27\}$ is forward sufficient and
cycle sufficient.\end{thm}
\begin{proof}
Consider the graph $\widetilde{\Gamma}_{27}$, drawn in Figure \ref{27}.

\begin{figure}[h]
\begin{centering}
\includegraphics{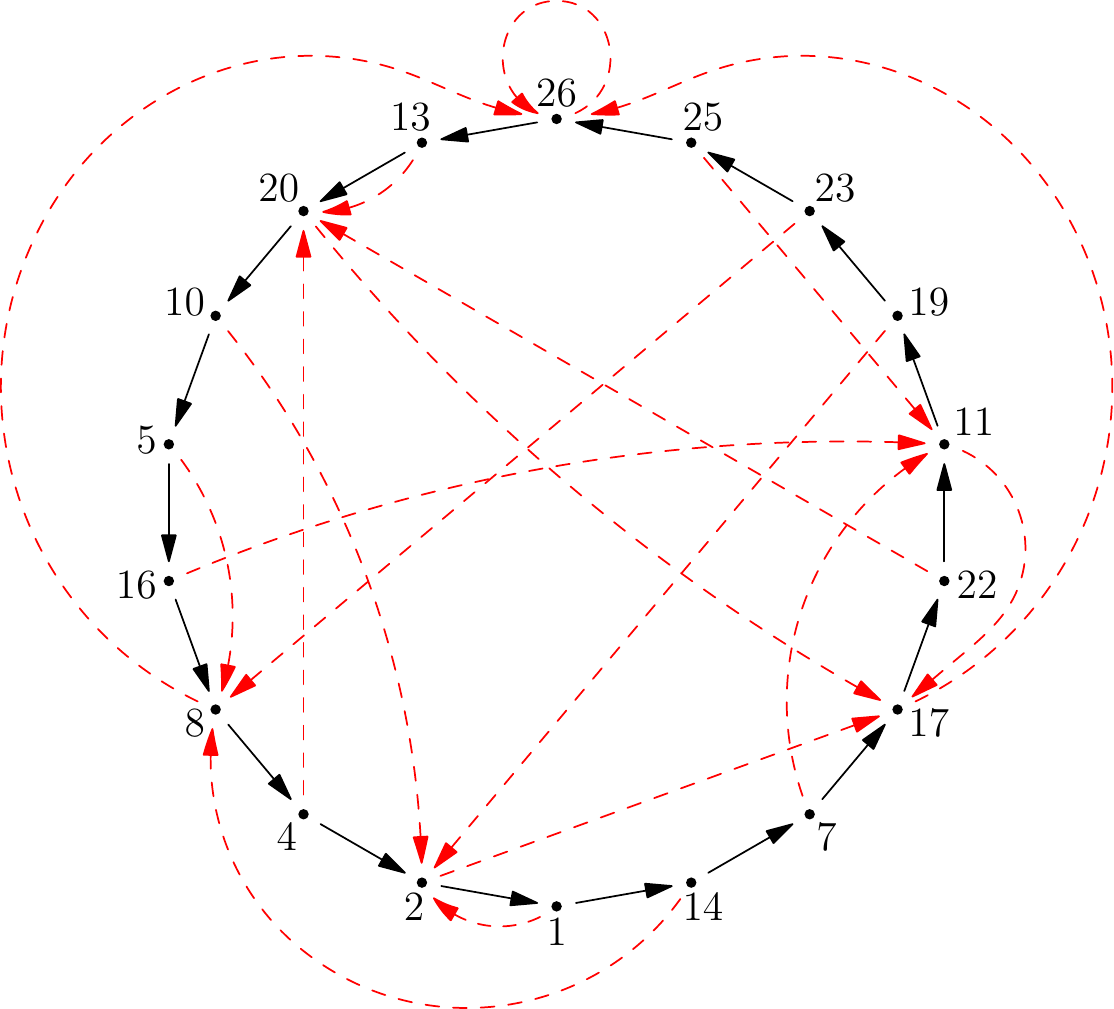} 
\par\end{centering}

\caption{\label{27} The action of $T_{0}$ and $T_{1}$ on the residues mod
$27$ relatively prime to $3$.}
\end{figure}

Now, suppose for contradiction that there is a nontrivial $T$-cycle
or divergent $T$-orbit of positive integers which does not contain
an integer congruent to $20$ mod $27$. Consider the path $P$ on
$\Gamma_{27}$ formed by taking this $T$-orbit mod $27$, starting
at the first element which is not divisible by $3$. Consider the
subgraph $\Gamma_{27}'$ of $\widetilde{\Gamma}_{27}$ formed by deleting
the node $20$ and all its adjacent edges. Since the path $P$ does
not contain the node $20$ by assumption, we see that $P$ lies entirely
within $\Gamma_{27}'$.

Notice that in $\Gamma_{27}'$, the node $10$ has no arrows coming
into it, so it cannot occur more than once in the path $P$. Similarly,
the nodes $5$, $16$, $13$, and $26$ cannot occur more than once
in $P$. Thus, some infinite tail $P'$ of the path $P$ must lie
in the subgraph $\Gamma_{27}''$ shown in Figure \ref{27no20simplified}.

\begin{figure}[h]
\begin{centering}
\includegraphics{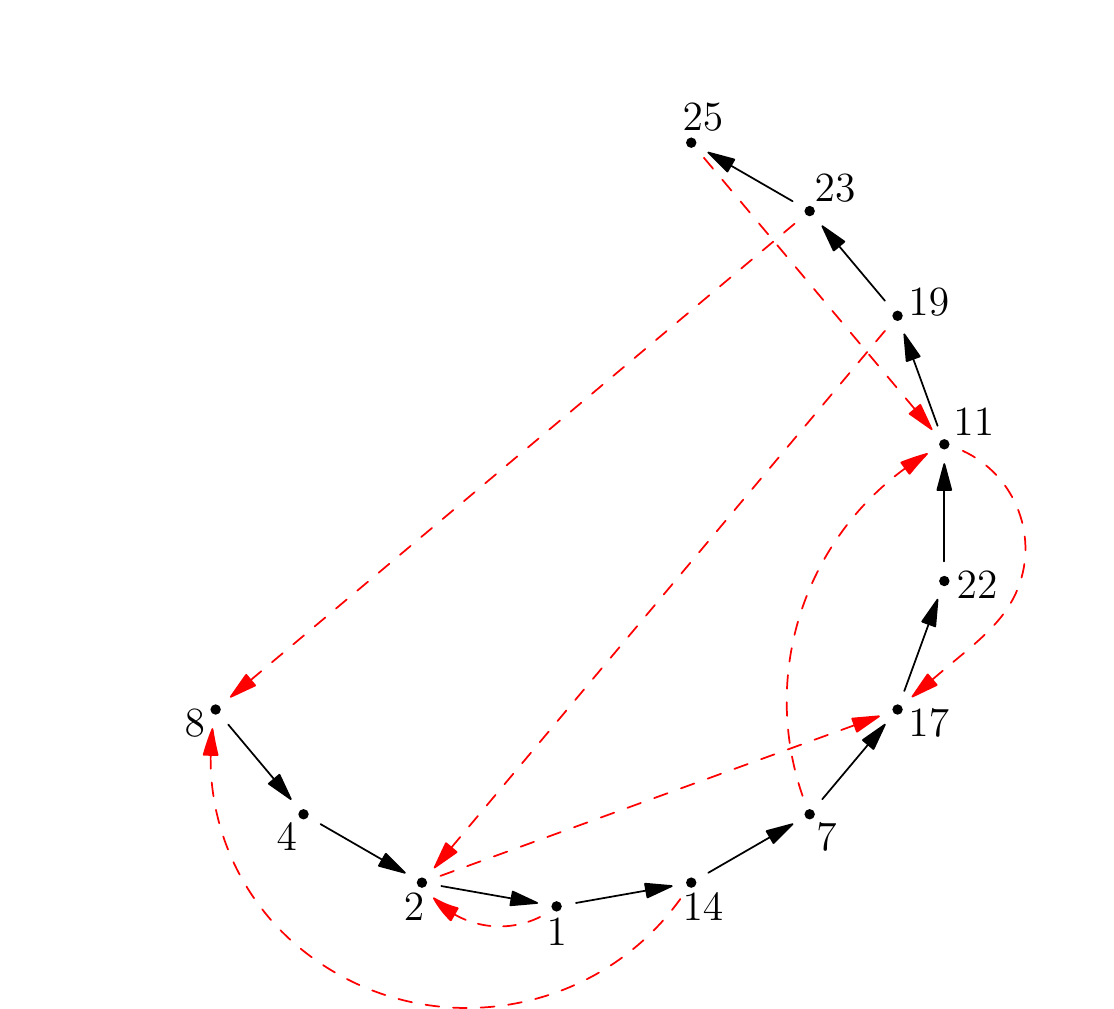} 
\par\end{centering}

\caption{\label{27no20simplified} The graph $\Gamma_{27}''$.}
\end{figure}

Note that by Proposition \ref{percentage}, any nontrivial cycle has
$m\ge3$ and hence its percentage of odd elements is at least $\ln(2)/\ln(3+1/3)\approx0.576$,
and the percentage of odd elements in any divergent $T$-orbit is
at least $0.6309$. We show that the fraction of red arrows followed
by any infinite path in $\Gamma_{27}''$ is at most $0.5$, hence
obtaining a contradiction.

We first note that any consecutive path of red arrows in $\Gamma_{27}''$
has length at most $2$. Moreover, any path of $2$ consecutive red
arrows (either from $19$ to $2$ to $17$ or from $1$ to $2$ to
$17$) must be followed by at least $2$ consecutive black arrows
(from $17$ to $22$ to $11$). It follows that the path $P'$ has
at most $50$ percent red arrows, as desired. 
\end{proof}
The key observation in the proof above is that the graph $\Gamma_{27}''$
essentially has too many black edges. We can use similar methods to
obtain simple graph-theoretic criteria for strong sufficiency in the
forward and backward directions and for cycle sufficiency.
\begin{defn*}
A \textit{simple cycle} in a directed graph is a directed path $A_{1},\ldots,A_{k}$
of nodes for which $A_{i}=A_{j}$ if and only if $\{i,j\}=\{1,k\}$.\end{defn*}
\begin{prop}
\label{pretzelcycles} Let $d\in\mathbb{N}$, and let $a_{1},\ldots,a_{k}$
be $k$ distinct residues mod $d$. Let $\Gamma_{d}'$ be the subgraph
of $\widetilde{\Gamma}_{d}$ formed by deleting the nodes labeled
$a_{1},\ldots,a_{k}$ and all arrows connected to them, and let $\Gamma_{d}''$
be the graph formed from $\Gamma_{d}'$ by deleting any edge which
is not contained in any cycle of $\Gamma_{d}'$. 
\begin{itemize}
\item [\emph{(a)}]If the fraction of red arrows in every simple cycle
of $\Gamma_{d}''$ is less than $\ln(2)/\ln(3)$, then $\{a_{1},\ldots,a_{k}\bmod d\}$
is forward sufficient. 
\item [\emph{(b)}]If the fraction of red arrows in every simple cycle
of $\Gamma_{d}''$ is greater than $\ln(2)/\ln(3)$, then $a_{1},\ldots,a_{k}\bmod d$
is backward sufficient. 
\item [\emph{(c)}]If the fractions of red arrows in the simple cycles
of each connected component $H$ of $\Gamma_{d}''$ are either all
greater than $\ln(2)/\ln(3)$ or all less than $\ln(2)/\ln(3+1/m)$
where $m=2^{60}$, then $a_{1},\ldots,a_{k}\bmod d$ is cycle sufficient. 
\end{itemize}
\end{prop}
\begin{proof}
Let $\rho=\ln(2)/\ln(3)$. First, suppose the fraction of red arrows
in every simple cycle of $\Gamma_{d}''$ is less than $\rho$. We
show that every infinite path in $\Gamma_{d}''$ must also have its
limiting fraction of red arrows less than $\rho$, showing that $a_{1},\ldots,a_{k}\bmod d$
is forward sufficient. Let $P$ be an infinite path $A_{1},A_{2},\ldots$
of nodes in $\Gamma_{d}''$.

Since $P$ is infinite and $\Gamma_{d}''$ has a finite number of
nodes, some node must occur infinitely many times in $P$. Call this
node $A$. We show that the fraction of red arrows in the portion
of $P$ between any two consecutive occurrences of $A$ is less than
$\rho$. Let $A,B_{1},\ldots,B_{n},A$ be such a sub-path of $P$,
and call this sub-path $X$.

To show that the fraction of red arrows along the path $X$ must be
less than $\rho$, we induct on an invariant which we call the \textit{complexity}
of $X$. Define the \textit{complexity} of $X$ to be the number of
pairs of equal nodes in the sequence $A,B_{1},\ldots,B_{n},A$. For
instance, the complexity of the sequence $A,B,C,B,C,A$ is $3$, and
the complexity of the sequence $A,B,C,D,B,C,D,B,D,A$ is $8$.

For the base case, suppose $X$ has complexity $1$. Then all of $B_{1},\ldots,B_{n}$
are distinct, and so $X$ is a simple cycle. By our hypothesis, the
fraction of red arrows in $X$ is less than $\rho$.

Let $n\ge1$, and assume for strong induction that if $X$ has complexity
at most $n$ then the fraction of red arrows in the path $X$ is less
than $\rho$. Suppose $X$ has complexity $n+1$. Choose a node $B$
other than $A$ which occurs twice in $X$. Then we can write $X=u,B,v,B,w$
for some sequences of nodes $u$, $v$, and $w$.

Now, notice that the complexity of the sub-path $B,v,B$ of $X$ is
strictly less than that of $X$, since it does not contain the two
copies of $A$ on each end. Letting $a$ be the number of red arrows
along this path and $b$ the total number of arrows, we have that
$a/b<\rho$ by the induction hypothesis.

Let $X'$ be the cyclic path formed by deleting this cycle from $X$
to form the sequence of nodes $u,B,w$. Then the complexity of $X'$
is also less than that of $X$, so if $c$ is the number of red arrows
along $X'$ and $e$ is the total number of arrows, we have that $c/e<\rho$
by the induction hypothesis.

Finally, we have that $(a+c)/(b+e)$ is the fraction of red arrows
in the entire path $X$. It is well-known that this \textit{Farey
sum}, also known as the \textit{mediant} of the fractions $a/b$ and
$c/e$, must lie between $a/b$ and $c/e$. Hence it must also be
less than $\rho$. This completes the induction, proving the first
claim.

The second claim is analogous. For the third claim, note that the
$3x+1$ conjecture has now been verified for the positive integers
less than $2^{60}$, so any nontrivial cycle must have its minimal
element $m$ and maximal element $M$ both greater than $2^{60}$.
Furthermore, any infinite periodic path lying in $\Gamma_{d}''$ must
lie entirely in one of the connected components of $\Gamma_{d}''$.

Assume that for all connected components $H$ of $\Gamma_{d}''$,
the fractions of red arrows in the simple cycles of $H$ are either
all greater than $\ln(2)/\ln(3)$ or all less than $\ln(2)/\ln(3+1/m)$
where $m=2^{60}$. Suppose to the contrary that there is an infinite
periodic path $P$ in $\Gamma_{d}''$, and let $H$ be the connected
component containing it. If the simple cycles in $H$ have fractions
of red arrows less than $\ln(2)/\ln(3+1/m)$, then by the above argument,
the fraction of red arrows in $P$ is also less than $\ln(2)/\ln(3+1/m)$,
contradicting Proposition \ref{percentage}. If instead the simple
cycles in $H$ have fractions of red arrows greater than $\ln(2)/\ln(3)$,
then the fraction of red arrows in $P$ is also greater than $\ln(2)/\ln(3)>\ln(2)/\ln(3+1/M)$,
again contradicting Proposition \ref{percentage}. This completes
the proof.
\end{proof}
Using Proposition \ref{pretzelcycles}, we have obtained, with the
use of a computer, several examples of forward sufficient, backward
sufficient, and cycle sufficient sets that do not appear in Table
\ref{loopsvalues}. We list these results in Tables \ref{blackvalues},
\ref{redvalues}, and \ref{bothvalues}.

\begin{table}
\begin{centering}
\begin{tabular}{rrrrr}
\toprule 
\multicolumn{5}{c}{\textbf{Forward sufficient sets}}\tabularnewline
\midrule 
$3\bmod4$  & $11,15\bmod16$  & $0,1,6\bmod13$  & $3,9,12\bmod14$  & $7,15,19\bmod20$ \tabularnewline
\midrule 
$5\bmod6$  & $4,13\bmod18$  & $0,2,3\bmod13$  & $3,9,13\bmod14$  & $9,11,15\bmod20$ \tabularnewline
\midrule 
$3\bmod8$  & $11,17\bmod18$  & $0,2,6\bmod13$  & $4,5,13\bmod14$  & $11,15,18\bmod20$ \tabularnewline
\midrule 
$6\bmod8$  & $13,17\bmod18$  & $0,3,9\bmod13$  & $4,11,13\bmod14$  & $11,15,19\bmod20$ \tabularnewline
\midrule 
$4\bmod9$  & $5,14\bmod21$  & $0,3,10\bmod13$  & $5,6,7\bmod14$  & $10,14,17\bmod21$ \tabularnewline
\midrule 
$8\bmod9$  & $5,17\bmod24$  & $0,6,9\bmod13$  & $5,6,9\bmod14$  & $13,14,17\bmod21$ \tabularnewline
\midrule 
$5\bmod12$  & $11,14\bmod24$  & $0,6,10\bmod13$  & $5,7,12\bmod14$  & $14,17,20\bmod21$ \tabularnewline
\midrule 
$8\bmod18$  & $11,17\bmod24$  & $0,8,9\bmod13$  & $5,7,13\bmod14$  & $3,10,17\bmod22$ \tabularnewline
\midrule 
$20\bmod27$  & $11,19\bmod24$  & $1,3,7\bmod13$  & $5,9,12\bmod14$  & $3,17,20\bmod22$ \tabularnewline
\midrule 
$0,3\bmod7$  & $14,20\bmod24$  & $1,3,11\bmod13$  & $5,9,13\bmod14$  & $3,17,21\bmod22$ \tabularnewline
\midrule 
$0,5\bmod7$  & $14,22\bmod24$  & $1,6,7\bmod13$  & $6,7,11\bmod14$  & $4,15,19\bmod22$ \tabularnewline
\midrule 
$1,7\bmod8$  & $14,23\bmod24$  & $1,6,11\bmod13$  & $6,9,11\bmod14$  & $5,16,17\bmod22$ \tabularnewline
\midrule 
$4,5\bmod11$  & $17,23\bmod24$  & $2,3,4\bmod13$  & $7,8,13\bmod14$  & $8,17,19\bmod22$ \tabularnewline
\midrule 
$4,8\bmod11$  & $10,17\bmod27$  & $2,3,11\bmod13$  & $7,11,12\bmod14$  & $12,13,19\bmod22$ \tabularnewline
\midrule 
$2,11\bmod12$  & $13,17\bmod27$  & $2,4,6\bmod13$  & $7,11,13\bmod14$  & $4,17,22\bmod24$ \tabularnewline
\midrule 
$7,10\bmod12$  & $13,22\bmod27$  & $2,6,11\bmod13$  & $8,9,13\bmod14$  & $7,17,20\bmod24$ \tabularnewline
\midrule 
$7,11\bmod12$  & $17,26\bmod27$  & $2,8,11\bmod13$  & $9,11,12\bmod14$  & $7,17,22\bmod24$ \tabularnewline
\midrule 
$5,11\bmod15$  & $22,26\bmod27$  & $3,7,9\bmod13$  & $9,11,13\bmod14$  & $7,19,20\bmod24$ \tabularnewline
\midrule 
$3,11\bmod16$  & $1,3,9\bmod10$  & $3,9,11\bmod13$  & $1,3,7\bmod16$  & $7,19,22\bmod24$ \tabularnewline
\midrule 
$6,7\bmod16$  & $1,5,9\bmod10$  & $6,7,9\bmod13$  & $1,3,9\bmod16$  & $7,19,23\bmod24$ \tabularnewline
\midrule 
$6,14\bmod16$  & $3,7,9\bmod10$  & $6,9,11\bmod13$  & $1,3,14\bmod16$  & $8,17,23\bmod27$ \tabularnewline
\midrule 
$7,9\bmod16$  & $5,7,9\bmod10$  & $3,6,7\bmod14$  & $2,7,12\bmod16$  & $8,17,25\bmod27$ \tabularnewline
\midrule 
$7,11\bmod16$  & $1,2,5\bmod11$  & $3,6,9\bmod14$  & $2,11,12\bmod16$  & $10,11,13\bmod27$ \tabularnewline
\midrule 
$9,12\bmod16$  & $1,2,8\bmod11$  & $3,7,10\bmod14$  & $2,12,14\bmod16$  & $10,11,26\bmod27$ \tabularnewline
\midrule 
$9,14\bmod16$  & $1,5,9\bmod11$  & $3,7,12\bmod14$  & $5,11,16\bmod18$  & \tabularnewline
\midrule 
$9,15\bmod16$  & $1,8,9\bmod11$  & $3,7,13\bmod14$  & $7,9,15\bmod20$  & \tabularnewline
\midrule 
$11,14\bmod16$  & $0,1,3\bmod13$  & $3,9,10\bmod14$  & $7,15,18\bmod20$  & \tabularnewline
\bottomrule
\end{tabular}
\par\end{centering}

\caption{\label{blackvalues} Some forward sufficient sets obtained using the
first criterion in Proposition \ref{pretzelcycles}.}
\end{table}

\begin{table}
\begin{centering}
\begin{tabular}{rrrr}
\toprule 
\multicolumn{4}{c}{\textbf{Backward sufficient sets }}\tabularnewline
\midrule 
$2,4\bmod8$  & $1,3,5\bmod16$  & $2,4,12\bmod16$  & $1,4,20\bmod24$ \tabularnewline
\midrule 
$2,5\bmod8$  & $1,3,8\bmod16$  & $2,5,12\bmod16$  & $1,5,13\bmod24$ \tabularnewline
\midrule 
$1,8\bmod12$  & $1,3,10\bmod16$  & $2,5,13\bmod16$  & $1,8,20\bmod24$ \tabularnewline
\midrule 
$2,4\bmod18$  & $1,4,12\bmod16$  & $2,8,12\bmod16$  & $2,4,20\bmod24$ \tabularnewline
\midrule 
$2,5\bmod24$  & $1,5,13\bmod16$  & $2,8,13\bmod16$  & $2,8,20\bmod24$ \tabularnewline
\midrule 
$1,4,10\bmod12$  & $1,8,13\bmod16$  & $2,10,12\bmod16$  & \tabularnewline
\midrule 
$1,3,4\bmod16$  & $2,3,10\bmod16$  & $1,4,10\bmod18$  & \tabularnewline
\bottomrule
\end{tabular}
\par\end{centering}

\caption{\label{redvalues} Some backward sufficient sets obtained using the
second criterion in Proposition \ref{pretzelcycles}.}
\end{table}

\begin{table}
\begin{centering}
{ %
\begin{tabular}{c}
\toprule {\textbf{Cycle sufficient sets}} \tabularnewline
\hline 
$1,3\bmod16$ \tabularnewline
\hline 
$2,12\bmod16$ \tabularnewline
\hline 
\end{tabular}} 
\par\end{centering}

\caption{\label{bothvalues} Some cycle sufficient sets obtained using the
third criterion in Proposition \ref{pretzelcycles}.}
\end{table}

\section{Self duality and folding in $\Gamma_{2^{n}}$}

\label{DualityFolding}

We now use properties of the $2$-adic dynamical system $T:\mathbb{Z}_{2}\to\mathbb{Z}_{2}$
to provide a better understanding of the graphs $\Gamma_{2^{n}}$.
We will use these insights to find more strongly sufficient sets from
the ones we have already found.

\subsection{Self color duality}

The graphs $\Gamma_{2^{n}}$ exhibit a surprising and beautiful self-duality.
\begin{defn*}
Let $\Gamma$ be any directed graph having each edge colored either
red or black. The \textit{color dual} of $\Gamma$ is the graph formed
by replacing all red edges with black edges and vice versa.
\end{defn*}

\begin{defn*}
A graph is \textit{self color dual} if it is isomorphic to its color
dual up to a relabeling of the vertices.
\end{defn*}
We give a complete classification of the self color dual graphs $\Gamma_{k}$.
\begin{thm}
\label{selfdual} The graph $\Gamma_{k}$ is self color dual if and
only if $k=2^{n}$ for some positive integer $n$.
\end{thm}
To prove this, we require some terminology and background. Define
$\mathbb{Z}_{2}$ to be the ring of $2$-adic integers equipped with
the usual $2$-adic metric. The map $T$ can be extended to be defined
on $\mathbb{Z}_{2}$. Define the \textit{parity vector function} 
\[
\Phi^{-1}:\mathbb{Z}_{2}\to\mathbb{Z}_{2}
\]
to be the map sending $x$ to the $T$-orbit of $x$ taken mod $2$.
Bernstein \cite{Bernstein} shows that the inverse parity vector function
$\Phi$ is well-defined, that is, the parity vector of a $2$-adic
uniquely determines the $2$-adic. Moreover, Lagarias \cite{Lagarias}
shows that $T$ is conjugate to the binary shift map 
\[
\sigma:\mathbb{Z}_{2}\to\mathbb{Z}_{2},
\]
the map sending a $2$-adic binary expansion $a_{0}a_{1}a_{2}a_{3}\ldots$
to the shifted $2$-adic $a_{1}a_{2}a_{3}\ldots$, via the parity
vector function $\Phi^{-1}$. That is, $T=\Phi\circ\sigma\circ\Phi^{-1}$.

In \cite{Hedlund}, Hedlund shows that there are exactly two continuous
\textit{autoconjugacies} of the shift map (conjugacies from $\sigma$
to $\sigma$), namely the identity map and the ``bit complement''
map $V:\mathbb{Z}_{2}\to\mathbb{Z}_{2}$ given by 
\[
V(a_{0}a_{1}a_{2}\ldots)=b_{0}b_{1}b_{2}\ldots
\]
 where $b_{i}=1-a_{i}$ for all $i$. For instance, $V(100100100\ldots)=011011011\ldots$.

In \cite{Dad}, the second author uses Hedlund's result to demonstrate
that there are exactly two continuous autoconjugacies of $T$ with
itself. The identity map is one such map. The other, denoted $\Omega:\mathbb{Z}_{2}\to\mathbb{Z}_{2}$,
is the map 
\[
\Omega:=\Phi\circ V\circ\Phi^{-1}.
\]
 We will use $\Omega$ to demonstrate self-color-duality in $\Gamma_{2^{n}}$.
We use the fact that $V$ is an involution, and hence $\Omega$ is
an involution as well, that is, $\Omega^{2}=1$. In particular, we
have that 
\[
T_{1}\circ\Omega=\Omega\circ T_{0}
\]
 and 
\[
\Omega\circ T_{0}=T_{1}\circ\Omega
\]
 where these maps are defined.

Finally, a map $f:\mathbb{Z}_{2}\to\mathbb{Z}_{2}$ is called \textit{solenoidal}
if it induces a permutation on $\mathbb{Z}/2^{n}\mathbb{Z}$ for all
$n$. It is known (\cite{Hedlund}, \cite{Lagarias}, \cite{Dad})
that the maps $V$, $\Phi$, $\Phi^{-1}$, and hence $\Omega$ are
all solenoidal. Note that $\Omega$ therefore induces an involution
on $\mathbb{Z}/2^{n}\mathbb{Z}$ as well.

We now have the tools to prove Theorem \ref{selfdual}.
\begin{proof}
Let $n\geq1$ and let $\Gamma_{2^{n}}^{\ast}$ denote the color dual
of $\Gamma_{2^{n}}$. Let $\Gamma_{2^{n}}^{\Omega}$ denote the graph
formed from $\Gamma_{2^{n}}$ by replacing each node label $a$ with
$\Omega(a)\bmod2^{n}$. We show that $\Gamma_{2^{n}}^{\Omega}=\Gamma_{2^{n}}^{\ast}$,
from which it follows that $\Gamma_{2^{n}}^{\ast}$ is isomorphic
to $\Gamma_{2^{n}}$ up to a relabeling of the nodes.

Suppose that in $\Gamma_{2^{n}}^{\ast}$, there is a red arrow from
$a$ to $b$. Then in $\Gamma_{2^{n}}$, there is a black arrow from
$a$ to $b$. It follows that there are positive integers $x$ and
$y$ congruent to $a$ and $b$ mod $2^{n}$ respectively for which
$T_{0}(x)=y$. Therefore $\Omega(T_{0}(x))=\Omega(y)$, and hence
$T_{1}(\Omega(x))=\Omega(y)$. Thus, in $\Gamma_{2^{n}}$, there is
a red arrow from $\Omega(a)$ to $\Omega(b)$. Since $\Omega$ is
an involution, in $\Gamma_{2^{n}}^{\Omega}$, there is a red arrow
from $\Omega(\Omega(a))=a$ to $\Omega(\Omega(b))=b$.

Similarly, if there is a black arrow from $a$ to $b$ in $\Gamma_{2^{n}}^{\ast}$
then there is a black arrow from $a$ to $b$ in $\Gamma_{2^{n}}^{\Omega}$.

For the reverse direction, suppose that in $\Gamma_{2^{n}}^{\Omega}$
there is a red arrow from $a$ to $b$. Then in $\Gamma_{2^{n}}$,
there is a red arrow from $\Omega(a)$ to $\Omega(b)$. Thus there
are positive integers $x$ and $y$ congruent to $a$ and $b$ mod
$2^{n}$ respectively for which $T_{1}(\Omega(x))=\Omega(y)$. Thus
$\Omega(T_{0}(x))=\Omega(y)$, and since $\Omega$ is an involution,
we have $T_{0}(x)=y$. It follows that there is a red arrow from $a$
to $b$ in $\Gamma_{2^{n}}^{\ast}$.

A similar argument shows that if there is a black arrow from $a$
to $b$ in $\Gamma_{2^{n}}^{\Omega}$ then there is a black arrow
from $a$ to $b$ in $\Gamma_{2^{n}}^{\ast}$. This shows that $\Gamma_{2^{n}}$
is self color dual. 

To prove that no other $\Gamma_{k}$ is self color dual, let $k=2^{n}b$
where $b$ is an odd positive integer greater than $1$ and assume
that $\Gamma_{2^{n}b}$ is self color dual. Then there exists a graph
isomorphism $\rho:\Gamma_{2^{n}b}\to\Gamma_{2^{n}b}$ mapping red
arrows to black ones and vice versa.

For any node $z$ in $\Gamma_{2^{n}b}$ define $\hat{T}_{0}\left(z\right)$
to be the set of nodes $w$ such that there is a black arrow from
$z$ to $w$ and $\hat{T}_{1}\left(z\right)$ to be the set of nodes
$w$ such that there is a red arrow from $z$ to $w$. Furthermore,
for any nonnegative integer $k$ define $\hat{T}^{k}\left(z\right)$
to be the set of nodes that can be reached starting from $z$ by a
path of length $k$. Clearly the graph isomorphism $\rho$ must preserve
the number of nodes that can be reached in such a manner, i.e. 
\begin{equation}
\left|\hat{T}^{k}\left(z\right)\right|=\left|\hat{T}^{k}\left(\rho\left(z\right)\right)\right|\label{eq:cardinality}
\end{equation}
 for any $z$ and $k$.

Suppose a node $z$ has a black arrow from $z$ to itself. Then by
the proof of Proposition \ref{GammaStructure}, if $n>0$ then $z$
is even and there exists an even integer $2a$ congruent to $z$ modulo
$2^{n}b$ such that either $T_{0}\left(2a\right)=a$ or $T_{0}\left(2a+2^{n}b\right)=a+2^{n-1}b$
is congruent to $z$, and thus to $2a$, modulo 2$^{n}b$. Thus either
$a\xcong{2^{n}b}0$ or $a\xcong{2^{n}b}2^{n-1}b$ so that in both
cases $2a$, and thus $z$, must be congruent to $0$ modulo $2^{n}b$.
A similar argument shows that the only node $z$ that has a red arrow
from $z$ to itself is $-1$. Since any color reversing graph isomorphism
must map these nodes to each other, $\rho\left(-1\right)=0$ and $\rho\left(0\right)=-1$.

We now show by finite induction that for any $k\in\left\{ 0,1,\ldots,n\right\} $,
$\hat{T}^{k}\left(-1\right)$ is the set of all nodes $z$ such that
$z\xcong{2^{n-k}b}-1$. For the base case, notice that $\hat{T}^{0}\left(z\right)=\left\{ z\right\} $
so that in particular $\hat{T}^{0}\left(-1\right)=\left\{ -1\right\} $,
i.e. the set of nodes that are congruent to $-1$ modulo $2^{n}b$.
If $n=0$ then we are done. If not, let $k<n$ and assume that $\hat{T}^{k}\left(-1\right)=\left\{ z\mid z\xcong{2^{_{n-k}}b}-1\right\} $
which is a set of odd nodes. Then $\hat{T}^{k+1}\left(-1\right)$
is the set of nodes obtained by following a red arrow from a node
$z\in\hat{T}^{k}\left(-1\right)$. Since $z\xcong{2^{n}b}-1+2^{n-k}bj$
for some $j$, and $T_{1}\left(-1+2^{n-k}bj\right)=-1+3\cdot2^{n-\left(k+1\right)}bj\xcong{2^{n-\left(k+1\right)}b}-1$
it follows that $z$ is in the set of all nodes that are congruent
to $-1$ modulo $2^{n-\left(k+1\right)}b$. Conversely if $w$ is
congruent to $-1$ modulo $2^{n-\left(k+1\right)}b$, then $w=-1+2^{n-\left(k+1\right)}bl$
for some $l$ and thus is congruent modulo $2^{n-k}b$ to 
\[
-1+2^{n-\left(k+1\right)}bl+2^{n-k}b=-1+3\cdot2^{n-\left(k+1\right)}bl=T_{1}\left(-1+2^{n-k}bl\right)
\]
Since $-1+2^{n-k}bl\xcong{2^{n-k}b}-1$ it is congruent to an element
of $\hat{T}^{k}\left(-1\right)$ and so there is a red arrow from
an element of $\hat{T}^{k}\left(-1\right)$ to $w$. Thus $\hat{T}^{k+1}\left(-1\right)$
is the set of nodes that are congruent to $-1$ modulo $2^{n-\left(k+1\right)}b$,
which completes the induction.

A similar argument shows that $\hat{T}^{k}\left(0\right)$ is the
set of all nodes $z$ such that $z\xcong{2^{n-k}b}0$ for all $k\in\left\{ 0,1,\ldots,n\right\} $.
Since the graph isomorphism $\rho$ must map the set of nodes that
are reachable by a path of length $n$ from $-1$ to the set of nodes
reachable by a path of length $n$ from $\rho\left(-1\right)=0$ we
have that $\rho$ maps the set of nodes congruent to $-1$ modulo
$b$ to those congruent to $0$ modulo $b$.

Now $b$ is odd, so $2$ is invertible modulo $b$. Let $z\xcong b-1$
and not congruent to $-1$ modulo $2b$. Then $z$ is even and $T_{0}\left(z\right)\xcong b-\frac{1}{2}$.
Conversely, if $w\xcong b-\frac{1}{2}$ then $w=T_{0}\left(z\right)$
for some even $z\xcong b-1$. Thus every node $z\xcong b-1$ that
is not congruent to $-1$ modulo $2b$ has a black arrow from $z$
to a node $w$ congruent to $-\frac{1}{2}$ modulo $b$ and every
node $w\xcong b-\frac{1}{2}$ has such an arrow pointing to it. Since
we have seen that all other nodes congruent to $-1$ modulo $b$ only
have arrows pointing to other such nodes, $\hat{T}^{n+1}\left(-1\right)$
consists of all nodes congruent to either $-1$ or $-\frac{1}{2}$
modulo $b$. 

Similar arguments show that $\hat{T}^{n+2}\left(-1\right)$ consists
of all nodes congruent to either $-1$, $-\frac{1}{2}$, or $-\frac{1}{4}$
modulo $b$ and that $\hat{T}^{n+2}\left(0\right)$ consists of all
nodes congruent to either $0$, $\frac{1}{2}$, $\frac{1}{4}$, or
$\frac{5}{4}$ modulo $b$. 

If $b$ is odd and greater than $5$, directly counting these nodes
shows that $\left|\hat{T}^{n+2}\left(-1\right)\right|=3\cdot2^{n}$
while $\left|\hat{T}^{n+2}\left(\rho\left(-1\right)\right)\right|=\left|\hat{T}^{n+2}\left(0\right)\right|=4\cdot2^{n}$
contradicting \ref{eq:cardinality}. If $b=3$, directly counting
these nodes shows that $\left|\hat{T}^{n+2}\left(-1\right)\right|=2\cdot2^{n}$
while $\left|\hat{T}^{n+2}\left(\rho\left(-1\right)\right)\right|=\left|\hat{T}^{n+2}\left(0\right)\right|=3\cdot2^{n}$
again contradicting \ref{eq:cardinality}. 

Finally, suppose $b=5$. Then $ $$\hat{T}^{n+1}\left(-1\right)\setminus\hat{T}^{n}\left(-1\right)$
consists of the nodes congruent to $-\frac{1}{2}$ modulo $5$ (i.e.
the nodes congruent to $2\bmod5$). All arrows from these nodes point
to a node congruent to $-\frac{1}{4}$ modulo $5$ (i.e. the nodes
congruent to $1\bmod5$). No node congruent to $1\bmod5$ is in $\hat{T}^{n}\left(-1\right)$
as these are all congruent to $-1\bmod5$. Similarly, $\hat{T}^{n+1}\left(0\right)\setminus\hat{T}^{n}\left(0\right)$
consists of the nodes congruent to $\frac{1}{2}$ modulo $5$ (i.e.
the nodes congruent to $3\bmod5$). 

Since $\rho$ is a graph isomorphism, it must map the set of nodes
$\hat{T}^{n}\left(-1\right)$ to $\hat{T}^{n}\left(0\right)$ and
$\hat{T}^{n+1}\left(-1\right)\setminus\hat{T}^{n}\left(-1\right)$
to $\hat{T}^{n+1}\left(0\right)\setminus\hat{T}^{n}\left(0\right)$
and also preserve the property that no arrow coming from a node in
$\hat{T}^{n+1}\left(0\right)\setminus\hat{T}^{n}\left(0\right)$ can
map to a node in $\hat{T}^{n}\left(0\right)$. But since $T(3)=5$,
the node $3$ has a red arrow mapping it to the node $5$, which is
in $\hat{T}^{n}\left(0\right)$ as these are all the nodes congruent
to $0\bmod5$. This is a contradiction, which completes the proof.
\end{proof}
To illustrate Theorem \ref{selfdual}, the graph $\Gamma_{8}$ is
shown in Figure \ref{mod8}, with each odd residue drawn directly
above its image under $\Omega$.
\begin{figure}
\begin{centering}
\includegraphics{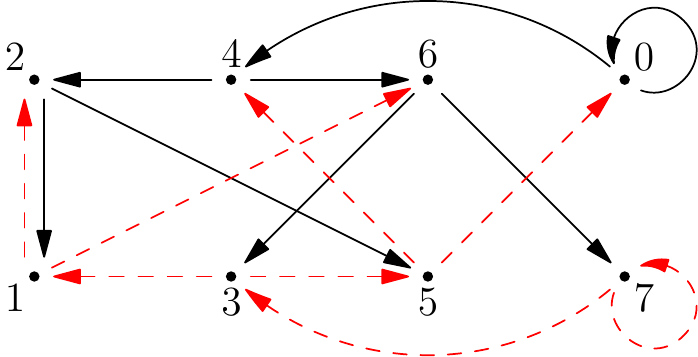} 
\par\end{centering}

\caption{\label{mod8} The digraph $\Gamma_{8}$. Self color duality is evident
by reflecting about the horizontal.}
\end{figure}

\subsection{Folding}

An \textit{endomorphism} of a map $f:\mathbb{Z}_{2}\to\mathbb{Z}_{2}$
is any map $h:\mathbb{Z}_{2}\to\mathbb{Z}_{2}$ for which $f\circ h=h\circ f$.
Note that an endomorphism is not necessarily invertible, and so while
all autoconjugacies of $T$ are endomorphisms of $T$, there may be
endomorphisms which are not autoconjugacies.

In \cite{Maria}, the fourth author classified and studied all continuous
endomorphisms of $T$ having solenoidal parity vector functions, and
in \cite{Keenan}, the first author and Kraft studied the remaining
continuous endomorphisms of $T$. It is natural to ask whether these
endomorphisms yield further insights into the structure of the graphs
$\Gamma_{2^{n}}$.

The simplest example of a continuous endomorphism of $T$ which is
not an autoconjugacy is defined in \cite{Maria} as follows. Let $D:\mathbb{Z}_{2}\to\mathbb{Z}_{2}$
be the \textit{discrete derivative} map, given by $D(a_{0}a_{1}a_{2}\ldots)=d_{0}d_{1}d_{2}\ldots$
where $d_{i}=|a_{i}-a_{i+1}|$ for all $i$. Then 
\[
R:=\Phi\circ D\circ\Phi^{-1}
\]
 is an endomorphism of $T$.

Unlike $\Omega$, the function $R$ is not solenoidal, since $D$
is not solenoidal. However, the value of $x$ mod $2^{n}$ determines
the value of $D(x)$ mod $2^{n-1}$ for all $n$. In particular, $D$
induces a $2$-to-$1$ map $\mathbb{Z}/2^{n}\mathbb{Z}\to\mathbb{Z}/2^{n-1}\mathbb{Z}$,
with $D(x)=D(V(x))$ for all $x$. Thus $R$ also induces a $2$-to-$1$
map $\mathbb{Z}/2^{n}\mathbb{Z}\to\mathbb{Z}/2^{n-1}\mathbb{Z}$,
with $R(x)=R(\Omega(x))$ for all $x$. We therefore obtain the following.
\begin{prop}
\label{folding} Let $n\ge2$ be a positive integer. 
\begin{itemize}
\item [\emph{(a)}]For any $x,y\in\mathbb{Z}/2^{n}\mathbb{Z}$, there is
a black edge between $R(x)\bmod2^{n-1}$ and $R(y)\bmod2^{n-1}$ in
$\Gamma_{2^{n-1}}$ if and only if there is a path of length two in
$\Gamma_{2^{n}}$ from $x$ to $y$ that consists of either two black
or two red edges. 
\item [\emph{(b)}]For any $x,y\in\mathbb{Z}/2^{n}\mathbb{Z}$, there is
a red edge between $R(x)\bmod2^{n-1}$ and $R(y)\bmod2^{n-1}$ in
$\Gamma_{2^{n-1}}$ if and only if there is a path of length two in
$\Gamma_{2^{n}}$ from $x$ to $y$ that consists of one black and
one red edge. 
\end{itemize}
\end{prop}
In other words, $\Gamma_{2^{n}}$ ``folds'' onto $\Gamma_{2^{n-1}}$
by identifying $\Omega$-pairs and using $D$ to define the edges.
For $n=3$, the graph $\Gamma_{8}$ shown in Figure \ref{mod8} can
be folded to obtain the graph $\Gamma_{4}$, by identifying the $\Omega$-pairs
of nodes and drawing in new edges according Proposition \ref{folding}.
(See Figure \ref{foldingexample}.)

\begin{figure}
\begin{centering}
\includegraphics{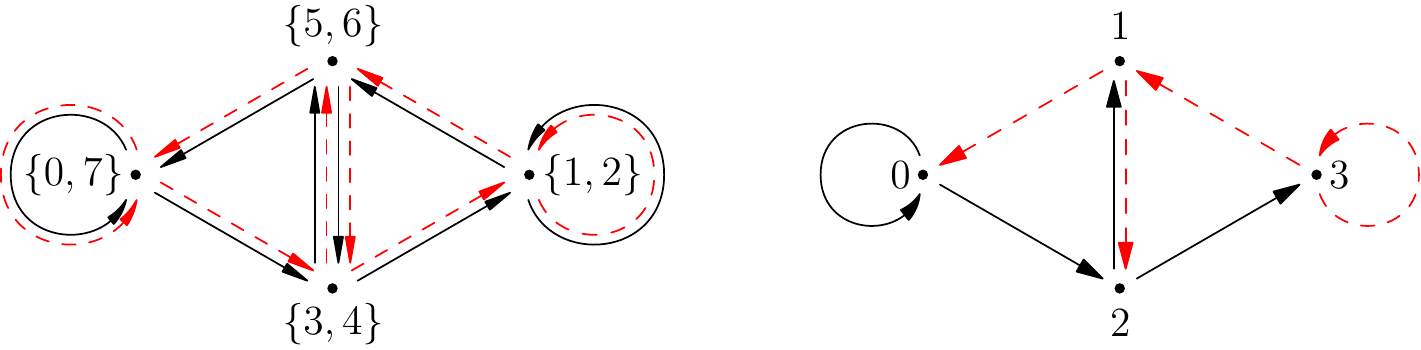} 
\par\end{centering}

\caption{\label{foldingexample} At left, the graph formed by identifying the
pairs of nodes in $\Gamma_{8}$ that map to each other under $\Omega$.
At right, the graph $\Gamma_{4}$.}
\end{figure}

More generally, we can fold the graphs $\Gamma_{2^{n}}$ onto any
$\Gamma_{2^{t}}$ for $t\le n$ in a similar manner using the endomorphisms
studied in \cite{Keenan}. For each $k\ge2$, define $M_{k}:\mathbb{Z}_{2}\to\mathbb{Z}_{2}$
to be the map given by $M_{k}(a_{0}a_{1}a_{2}\ldots)=m_{0}m_{1}m_{2}\ldots$
where 
\[
m_{i}=a_{i}+a_{i+1}+\cdots+a_{i+k-1}\bmod2
\]
 for all $i$. Then 
\[
H_{M_{k}}:=\Phi\circ M_{k}\circ\Phi^{-1}
\]
 is an endomorphism of $T$. Note that $M_{2}=D$ and $H_{M_{2}}=R$.

The aim of this section is to prove the following result, which enables
us to obtain more strongly sufficient sets modulo powers of $2$.
\begin{thm}
\label{MainFolding} Suppose $a_{1},a_{2},\ldots,a_{l}\bmod2^{n}$
satisfy the criterion for strong sufficiency of Proposition \ref{strongsufficiency}
for $d=2^{n}$. Let $q$ be the length of the largest cycle in $\Gamma''_{2^{n}}$,
and let $k$ be any positive integer satisfying $kq\le630,\!138,\!897$.
Then the preimage of $\{a_{1},a_{2},\ldots,a_{l}\}$ under $H_{M_{k}}$
modulo $2^{n+k-1}$ also satisfies the criterion from Proposition
\ref{strongsufficiency}, and is therefore a strongly sufficient set.\end{thm}
\begin{example*}
In Table \ref{loopsvalues}, we see that $1\bmod4$ satisfies the
criterion for strong sufficiency of Proposition \ref{strongsufficiency}.
We see from Figure \ref{foldingexample} that the inverse image of
$\{1\bmod4\}$ under $H_{M_{2}}$ is the set $\{3,4\bmod8\}$, which
is therefore also strongly sufficient.

Indeed, $\{3,4\bmod8\}$ also appears in Table \ref{loopsvalues}.
We can therefore unfold this set another time under $H_{M_{2}}$,
which shows that $\{7,8,9,10\bmod16\}$ is also strongly sufficient.
\end{example*}
To prove Theorem \ref{MainFolding}, we introduce the notation established
in \cite{Keenan}, for any $a\in\mathbb{Z}_{2}$, write $[a]$ to
denote the equivalence class of $a$ under the equivalence relation
$a\sim b$ if and only if $H_{M_{k}}(a)=H_{M_{k}}(b)$, i.e. $\left[a\right]=H_{M_{k}}^{-1}\left(\left\{ H_{M_{k}}\left(a\right)\right\} \right)$.
Notice that $H_{M_{k}}$ restricts to a well-defined surjective map
$\overline{{H}}_{M_{k}}:\mathbb{Z}/2^{n+k-1}\mathbb{Z}\to\bbZ/2^{n}\bbZ$
for any $n$ and $k>0$. We also use $\overline{a}$ to denote either
the congruence class of $a\in\bbZ_{2}$ modulo $2^{n}$ or $2^{n+k-1}$
when the power of $2$ is understood, and using this notation we have
$\overline{H_{M_{K}}(a)}=\overline{H}_{M_{k}}(\overline{a})$. Hence,
$\sim$ also restricts to an equivalence relation $\overline{\sim}$
on $\bbZ/2^{n+k-1}\bbZ$, in which residues $\overline{a}$ and $\overline{b}$
are equivalent if and only if $\overline{H}_{M_{k}}(\overline{a})=\overline{{H}}_{M_{k}}(\overline{b})$.
We also denote the $\overline{\sim}$-equivalence class of $\overline{a}$
by $[\overline{a}]$, i.e. $[\overline{a}]=\overline{H}_{M_{k}}^{-1}\left(\left\{ \overline{H}_{M_{k}}\left(\overline{a}\right)\right\} \right)$.
Notice that with this notation we have $\left[\overline{a}\right]=\overline{\left[a\right]}$.

Throughout this section, we write $x\rat y$ to indicate that there
is an arrow (either red or black) from $x$ to $y$ in the digraph
$\mathcal{G}$, $\Gamma_{2^{n+k-1}}$, or $\Gamma_{2^{n}}$.
\begin{lem}
\label{KeenLemma1} There is an arrow $\overline{H}_{M_{k}}(\overline{x})\rat\overline{H}_{M_{k}}(\overline{y})$
in $\Gamma_{2^{n}}$ if and only if there exist $\overline{a}\in[\overline{x}]$
and $\overline{b}\in[\overline{y}]$ for which $\overline{a}\rat\overline{b}$
in $\Gamma_{2^{n+k-1}}$. In this situation such arrows between the
elements of $\left[\overline{x}\right]$ and $\left[\overline{y}\right]$
form a bijection between $\left[\overline{x}\right]$ and $\left[\overline{y}\right]$.\end{lem}
\begin{proof}
Suppose $\overline{a}\in[\overline{x}]$ and $\overline{b}\in[\overline{y}]$
such that $\overline{a}\rat\overline{b}$ in $\Gamma_{2^{n+k-1}}$.
Then there exist $a,b\in\bbZ_{2}$ in the congruence classes $\overline{a}$
and $\overline{b}$ modulo $2^{n+k-1}$ respectively, $ $with $T(a)=b$.
Then since $H_{M_{k}}$ is an endomorphism of $T$, 
\begin{eqnarray*}
T\left(H_{M_{k}}\left(a\right)\right) & = & H_{M_{k}}\left(T\left(a\right)\right)\\
 & = & H_{M_{k}}\left(b\right).
\end{eqnarray*}
Thus $H_{M_{k}}\left(a\right)\rat H_{M_{k}}\left(b\right)$ in $\mathcal{G}$
and thus $\overline{H}_{M_{k}}(\overline{a})\rat\overline{H}_{M_{k}}(\overline{b})$
in $\Gamma_{2^{n}}$. 

Conversely, suppose $\overline{H}_{M_{k}}(\overline{x})\rat\overline{H}_{M_{k}}(\overline{y})$
in $\Gamma_{2^{n}}$. Then $H_{M_{k}}\left(x\right)\rat H_{M_{k}}\left(y\right)$
in $\mathcal{G}$. Therefore $T\left(H_{M_{k}}\left(x\right)\right)=H_{M_{k}}\left(y\right)$.
Thus $H_{M_{k}}\left(T\left(x\right)\right)=H_{M_{k}}\left(y\right)$
and thus $\left[y\right]=\left[T\left(x\right)\right]$. So taking
$a=x$ and $b=T\left(x\right)$ we have $a\rat b$ in $\mathcal{G}$
and consequently $\overline{a}\rat\overline{b}$ in $\Gamma_{2^{n+k-1}}$
and $\overline{a}\in[\overline{a}]=[\overline{x}]$ and $\overline{b}\in\left[\overline{T\left(x\right)}\right]=\left[\overline{y}\right]$.

Let $s,t$ be nodes in $\Gamma_{2^{n+k-1}}$ such that $s\rat t$.
Then $s=\overline{a}$ for some $a$ and $t=\overline{T\left(a\right)}$.
Since $T$ restricts to a bijection between $\left[a\right]$ and
$\left[T\left(a\right)\right]$ by the proof of Lemma 23 in \cite{Keenan}
it induces a bijection from $\overline{\left[a\right]}=\left[s\right]$
to $\overline{\left[T\left(a\right)\right]}=\left[t\right]$ in $\Gamma_{2^{n+k-1}}$.\end{proof}
\begin{lem}
\label{KeenLemma4} Suppose $x_{1},\ldots,x_{j}$ are nodes in $\Gamma_{2^{n}}$
such that the subgraph induced by these nodes is a cycle. Then the
subgraph induced by $\overline{H}_{M_{k}}^{-1}\left(\left\{ x_{1},\ldots,x_{j}\right\} \right)$
is a union of disjoint cycles in $\Gamma_{2^{n+k-1}}$. \end{lem}
\begin{proof}
Consider a sequence of nodes $x_{1},\ldots,x_{j}$ in $\Gamma_{2^{n}}$
that form a cycle in $\Gamma_{2^{n}}$, such that the only arrows
between the $x_{i}$'s are the arrows forming the cycle. Then by Lemma
\ref{KeenLemma1}, the arrows from nodes of $\overline{H}_{M_{k}}^{-1}\left(\left\{ x_{i}\right\} \right)$
and $\overline{H}_{M_{k}}^{-1}\left(\left\{ x_{i+1}\right\} \right)$
form a bijection between these sets for any $1\le i\le j$ (where
we set $x_{j+t}:=x_{t}$ for convenience). 

Now, given a node $z$ in one of the sets $\overline{H}_{M_{k}}^{-1}\left(\left\{ x_{i}\right\} \right)$,
there is exactly one arrow $z\rat z_{1}$ for some $z_{1}\in\overline{H}_{M_{k}}^{-1}\left(\left\{ x_{i+1}\right\} \right)$.
Moreover, since the only arrows between the $x_{i}$'s are the arrows
forming the cycle, there are no arrows from $z$ to any other node
in $\overline{H}_{M_{k}}^{-1}\left(\left\{ x_{1},\ldots,x_{j}\right\} \right)$.
Similarly, there is a unique arrow $z_{1}\rat z_{2}$ for some $z_{2}\in\overline{H}_{M_{k}}^{-1}\left(\left\{ x_{i+2}\right\} \right)$,
and there are no other arrows from $z_{1}$ into $\overline{H}_{M_{k}}^{-1}\left(\left\{ x_{1},\ldots,x_{j}\right\} \right)$.S
We continue this process to define a sequence of nodes $z\rat z_{1}\rat z_{2}\rat z_{3}\rat\cdots$.

Since there are only a finite number of nodes in $\bigcup_{i=1}^{j}\overline{H}_{M_{k}}^{-1}\left(\left\{ x_{i}\right\} \right)$,
the sequence $z,z_{1},z_{2},z_{3}\ldots$ must be eventually repeating,
say with minimum period $m$. Suppose $z_{t}\neq z$ is the first
entry at which the sequence repeats. Then both $z_{t-1}\rat z_{t}$
and $z_{t-1+m}\rat z_{t}$, and so $z_{t-1}$and $z_{t-1+m}$ must
both lie in $\overline{H}_{M_{k}}^{-1}\left(\left\{ x_{i+t-1}\right\} \right)$.$ $
But $z_{t-1}\neq z_{t-1+m}$ by minimality, contradicting Lemma \ref{KeenLemma1}.
It follows that $z_{t}=z$, and so the sequence $z,z_{1},z_{2},z_{3}\ldots$
is a cycle.

Similarly, choosing a node $z'$ not in this cycle, there is a cycle
$z'\rat z'_{1}\rat z'_{2}\rat z'_{3}\rat\cdots$ that must be disjoint
from the previous cycle. Continuing in this manner, we see that $\overline{H}_{M_{k}}^{-1}\left(\left\{ x_{1},\ldots,x_{j}\right\} \right)$
is a union of disjoint cycles in $\Gamma_{2^{n+k-1}}.$
\end{proof}
We now have the tools to prove our main result on folding.
\begin{proof}
{[}Proof of Theorem \ref{MainFolding}{]} Suppose $a_{1},a_{2},\ldots,a_{l}\bmod2^{n}$
satisfy the criterion for strong sufficiency of Proposition \ref{strongsufficiency}
for $d=2^{n}$. Let $\Gamma_{2^{n+k-1}}''$ be the graph formed by
deleting the preimage of $\{a_{1},a_{2},\ldots,a_{l}\}$ under $\overline{H}_{M_{k}}$
and all edges which are attached to those nodes, followed by the edges
which are not part of a cycle in the remaining graph. Since the edges
that remain are part of cycles, the cycles containing them map to
cycles in $\Gamma_{2^{n}}''$. Since the inverse image of such a cycle
is a disjoint union of cycles by Lemma \ref{KeenLemma4}, $\Gamma_{2^{n+k-1}}''$
consists of a disjoint union of cycles.

Let $q$ be the length of the largest cycle in $\Gamma''_{2^{n}}$,
and suppose $kq\le630,\!138,\!897$. By Corollary 22 in \cite{Keenan},
none of the cycles in $\Gamma_{2^{n+k-1}}''$ has length greater that
$kq$. Thus the maximum cycle in $\Gamma_{2^{n+k-1}}''$ has length
less than $630,\!138,\!897$. Thus $H_{M_{k}}^{-1}(\{a_{1},a_{2},\ldots,a_{l}\})$
is a set of nodes that satisfies the criterion of Proposition \ref{strongsufficiency},
and is therefore a strongly sufficient set.
\end{proof}

\section{Example: $1,3\bmod16$}

We conclude with an example that illustrates and links several of
the main results in this paper. In Table \ref{bothvalues}, we see
that $1,3\bmod16$ and $2,12\bmod16$ are sets that satisfy the third
criterion of Proposition \ref{pretzelcycles}. Figure \ref{1and3mod16}
shows $\Gamma_{16}''$ for the set $1,3\bmod16$.

\begin{figure}
\begin{centering}
\includegraphics{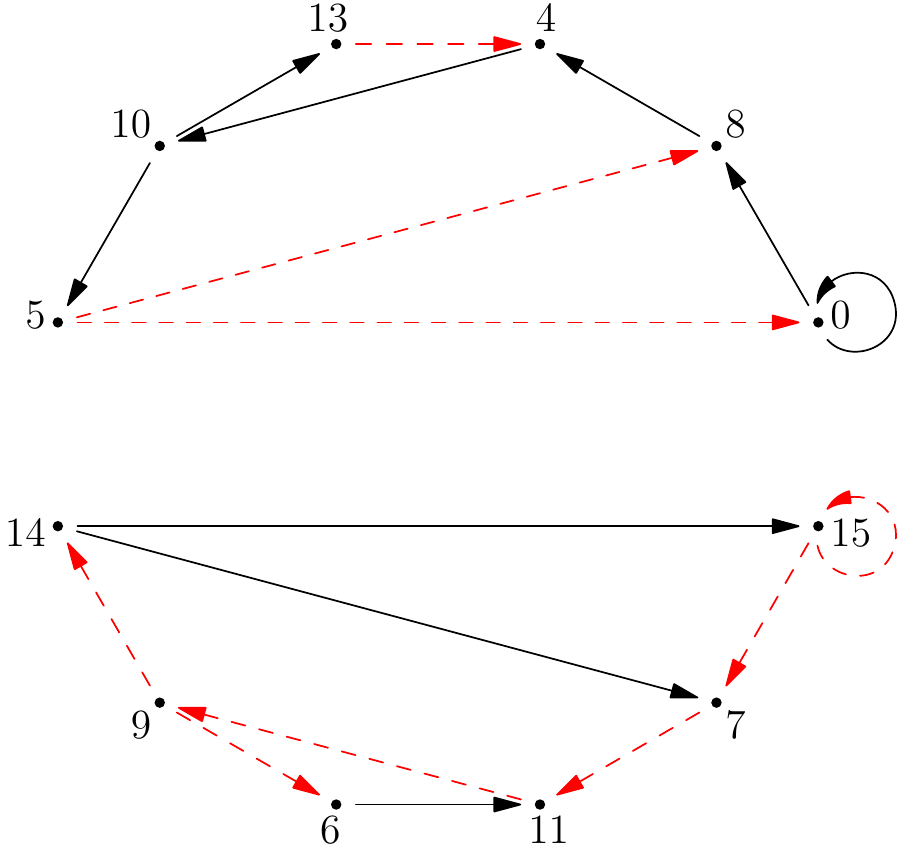} 
\par\end{centering}

\caption{\label{1and3mod16} The graph $\Gamma_{16}''$ obtained after removing
the nodes $1$ and $3$ from $\Gamma_{16}$ and then removing any
nodes or edges that are not contained in a cycle.}
\end{figure}

Notice that the connected component containing $0$ has the property
that every red arrow must be followed by at least two black arrows,
and the connected component containing $15$ has the opposite property:
every black arrow must be followed by at least two red arrows in any
infinite path. Hence, it does indeed satisfy the third criterion of
Proposition \ref{pretzelcycles}, and so every nontrivial cycle must
contain an element congruent to $1$ or $3$ mod $16$, i.e. $1,3\bmod16$
is a cycle sufficient set.

Notice further that the components exhibit the self color duality
in $\Gamma_{16}$: the connected component of $0$ maps to the connected
component of $15$ under $\Omega$, and in fact one component can
be reflected onto the other, with the colors of the arrows reversed,
matching each node with its $\Omega$-dual.

Finally, notice that $\Omega(1)\overmod{16}2$ and $\Omega(3)\overmod{16}12$.
By the self color duality of $\Gamma_{16}$, it follows that removing
the nodes $2$ and $12$ from $\Gamma_{16}$, and then removing the
nodes and edges not contained in any cycle, results in the same graph
$\Gamma_{16}''$ shown in Figure \ref{1and3mod16}. Thus $2,12\bmod16$
is a cycle sufficient set as well.

\section{Acknowledgments}

The authors would like to thank Gina Monks for her love and support
throughout this research project.


\begin{thebibliography}{10}
\bibitem{ApplegateLagarias} Applegate, D.~and Lagarias, J.~C.,
\textit{Density Bounds for the $3x+1$ Problem I. Tree-Search Method},
Math. Comp. 64 (1995), pp. 411-426. 

\bibitem{Bernstein} Bernstein, D.~J., \textit{A non-iterative $2$-adic
statement of the $3x+1$ conjecture}, Proc. Amer. Math. Soc. \textbf{121}
(1994), 405-408. 

\bibitem{Eliahou} Eliahou, S., \textit{The $3x+1$ problem: new lower
bounds on nontrivial cycle lengths}, Discrete Math. \textbf{188} (1993),
45-56. 

\bibitem{Hedlund} Hedlund, G., \textit{Endomorphisms and automorphisms
of the shift dynamical system}, Math. Systems Theory \textbf{3} (1969),
320-375. 

\bibitem{Hua} Hua, L.~K., \textit{Introduction to Number Theory},
Springer-Verlag, 1982, ISBN: 3-540-10818-1. 

\bibitem{Keenan} Kraft, B., Monks, K., \textit{On Conjugacies of
the $3x+1$ Map Induced by Continuous Endomorphisms of the Shift Dynamical
System}, Discrete Math. \textbf{310} (2010), 1875-1883. 

\bibitem{Lagarias} Lagarias, J.~C., \textit{The $3x+1$ problem
and its generalizations}, Am. Math. Monthly \textbf{92} (1985), 3-23. 

\bibitem{Dad} Monks, K.~G., Yasinski, J., \textit{The Autoconjugacy
of the $3x+1$ function}, Discrete Math. \textbf{275} (2004), 219-236. 

\bibitem{Kenny} Monks, K.~M., \textit{The sufficiency of arithmetic
progressions for the $3x+1$ conjecture}, Proc. Amer. Math. Soc.,
134 (10), October (2006), 2861-2872. 

\bibitem{Maria} Monks, M., \textit{Endomorphisms of the shift dynamical
system, discrete derivatives, and applications}, Discrete Math. \textbf{309}
(2009), 5196-5205. 

\bibitem{Sinisalo} Sinisalo, M.~K., \textit{On the minimal cycle
lengths of the Collatz sequences}, preprint, Univ. of Oulu, Finland,
2003. 

\bibitem{Wirsching} Wirsching, G., \textit{The Dynamical System Generated
by the $3n+1$ Function}, Lecture Notes in Math. 1681, Springer-Verlag,
1998, ISBN: 3-540-63970-5. \end{thebibliography}
\end{document}